\documentclass[11pt]{article}
\usepackage{amsfonts}
\usepackage{latexsym,amssymb}
\usepackage{amsmath, amsbsy}
\usepackage{amsopn, amstext}
\usepackage{graphicx, color, epstopdf}
\usepackage{threeparttable}

\newtheorem{theorem}{Theorem}

\numberwithin{equation}{section}
 \numberwithin{Lem}{section}
 \numberwithin{Defi}{section}
 \numberwithin{Theo}{section}
 \numberwithin{Rem}{section}
  \numberwithin{Coro}{section}
  \numberwithin{Fig}{section}

\def\NN{\hbox{\rlap{I}\kern.16em N}}
\def\NC{\hbox{\rlap{\kern.24em\raise.1ex\hbox
                  {\vrule height1.3ex width.9pt}}C}}

\title{An efficient spectral-Galerkin approximation and error analysis for Maxwell transmission eigenvalue problems in spherical geometries
\thanks{This work was supported in part by the National Natural Science Foundation of China (No. 11661022, 11471031, 91430216), NASF U1530401,
and the US National Science Foundation grant DMS-1419040.}}
\author{Jing An\thanks{Beijing Computational Science Research Center, Beijing 100193, China ({\tt anjing@csrc.ac.cn});
and School of Mathematical Sciences, Guizhou Normal University, Guiyang 550025, China ({\tt aj154@163.com}).}
\and Zhimin Zhang\thanks{Beijing Computational Science Research Center, Beijing 100193, China ({\tt zmzhang@csrc.ac.cn});
and Department of Mathematics, Wayne State University, Detroit, MI 48202, USA ({\tt zzhang@math.wayne.edu}).}}
\date{}

\begin{document}

\maketitle

\begin{abstract}
We propose and analyze an efficient spectral-Galerkin approximation for the Maxwell transmission eigenvalue problem in spherical geometry. Using a vector spherical harmonic expansion, we reduce the problem to a sequence of equivalent one-dimensional TE and TM modes that can be solved individually in parallel. For the TE mode, we derive associated generalized eigenvalue problems and corresponding pole conditions. Then we introduce weighted Sobolev spaces based on the pole condition and prove error estimates for the generalized eigenvalue problem. The TM mode is a coupled system with four unknown functions, which is challenging for numerical calculation. To handle it, we design an effective algorithm using Legendre-type vector basis functions. Finally, we provide some numerical experiments to validate our theoretical results and demonstrate the efficiency of the algorithms.

\vskip 5pt \noindent {\bf Keywords:} {Maxwell transmission eigenvalue problems, spherical geometry, TE and TM modes, spectral-Galerkin approximation, error analysis}


\end{abstract}

\section{Introduction}\label{int}
\label{Sec:Intro}
\indent The Maxwell transmission eigenvalue problem is a boundary value problem in a bounded domain which arises in inverse scattering theory for inhomogeneous media. It is well known that the transmission eigenvalues  play a critical role in the reconstruction of inhomogeneous non-absorbing media \cite{qua,TL,CO}.
 In addition, the transmission eigenvalues
are also used to estimate the index of refraction of a non-absorbing inhomogeneous medium
in recent years\cite{OT,TI,ES}. The method consists of several steps. First of all, the support of the
scattering obstacle can be recovered by using the measured scattering data and the linear
sampling method \cite{TIT} and the transmission eigenvalues can be identified from either the far
field or near field data \cite{EOT,OT}. Then, the bounds for smallest and largest eigenvalues of the (matrix) index of
refraction can be obtained in terms of the support of the scattering obstacle and the first
transmission eigenvalue of the anisotropic media \cite{OTD}. Finally, reconstructions of the electric
permittivity (if it is a scalar constant) or an estimate of the eigenvalues of the matrix
in the case of anisotropic permittivity can be obtained \cite{TIT}.

However, the effectiveness of the above method rests on having an efficient and robust algorithm
for computing Maxwell transmission eigenvalues.
Although the Maxwell transmission eigenvalue problem is stated in a simple form, its solution is complex since it is nonstandard so that the classical theory can not be applied directly. There are only a few papers dealing with the numerical computation of Maxwell transmission eigenvalues due to this notorious difficulty. In \cite{FE}, Monk et al. propose two finite element methods in computing a few lowest Maxwell's transmission eigenvalues,  curl-conforming finite element and mixed finite element methods.
In \cite{CO}, Sun et al. presented an iterative method to compute the Maxwell's transmission eigenvalue, where the transmission eigenvalue problem is written as a quad-curl eigenvalue problem. Then the real transmission eigenvalues are shown to be the roots of a nonlinear function whose value is the generalized eigenvalue of a related self-adjoint quad-curl eigenvalue problem which is computed by using a mixed finite element method. In \cite{AR}, Huang et al. also proposed a numerical algorithm for computing Maxwell transmission eigenvalue problems.  However, all these methods are based on low-order finite element methods, so they become very expensive if high accuracy is needed. To the best of our knowledge, no error analysis is given yet for any existing numerical method, even for three-dimensional spherical domains.

In practical applications, we often need to solve Maxwell transmission eigenvalue problems on the special domain of spherical geometries.
 We know of only a few reports on spectral-Galerkin approximation  for the Maxwell transmission eigenvalue problems on the special domain of spherical geometries. Thus, we will
present in this paper an efficient spectral-Galerkin approximation for the Maxwell
transmission eigenvalue problem in spherical geometry.
Using a vector spherical harmonic expansion, we reduce the problem to a sequence of equivalent one-dimensional
TE and TM modes that can be solved individually in parallel. For the TE mode, we
derive associated generalized eigenvalue problems and corresponding pole conditions.
Then we introduce weighted Sobolev spaces corresponding to the pole condition and prove
error estimates for the generalized eigenvalue problem. The TM mode is a coupled
system with four unknown functions, which is challenging for numerical calculation.
To handle it, we design an effective algorithm using Legendre-type vector basis
functions. Finally, we provide some numerical experiments to validate our theoretical
results  and demonstrate the efficiency of the algorithms.

\indent The rest of this paper is organized as follows. In the next section, we introduce the Maxwell transmission eigenvalue problem.
In \S3, we derive a dimension reduction scheme under spherical geometries.
In \S4, we derive the weak formulation and error estimate for the TE mode. In \S5, we describe in detail an efficient implementation of the algorithm.
We present several numerical experiments in \S6 to demonstrate the accuracy and efficiency of our method.
Finally, in \S7 we give some concluding remarks.

\section{Maxwell transmission eigenvalue problem }\label{Sec:Maxwell}
 \indent Let $D\subset \mathbb{R}^3$ be a bounded simply connected region with piecewise smooth boundary $\partial D$ and denote by $\pmb\nu$ the outward normal vector to  $\partial D$.
Let $(.,.)_D$ denote the scalar product in $L^2(D)^3$ and define the Hilbert spaces

\begin{align*}
&H(\mathbf{curl},D):=\{\mathbf{u}\in {L^2(D)}^3:\mathbf{curl}\; \mathbf{u}\in L^2(D)^3\},\\
&H_0(\mathbf{curl},D):=\{\mathbf{u}\in H(\mathbf{curl},D):\mathbf{u}\times \pmb\nu=\pmb 0 ~\text{on}~ \partial D\}
\end{align*}
equipped with the scalar product $(\mathbf{u},\mathbf{v})_{\mathbf{curl}}=(\mathbf{u},\mathbf{v})_{D}+(\mathbf{curl}\;\mathbf{u},\mathbf{curl}\; \mathbf{v})_{D}$ and the corresponding norm $\|\cdot\|_{\mathbf{curl}}$.
Following \cite{OTE}, we also define
\begin{align*}
&\mathcal{U}(D):=\{\mathbf{u}\in  H(\mathbf{curl},D):\mathbf{curl}\; \mathbf{u}\in  H(\mathbf{curl},D)\},\\
&\mathcal{U}_0(D):=\{\mathbf{u}\in H_0(\mathbf{curl},D):\mathbf{curl}\; \mathbf{u}\in  H_0(\mathbf{curl},D)\}
\end{align*}
equipped with the scalar product $(\mathbf{u},\mathbf{v})_{\mathcal{U}}=(\mathbf{u},\mathbf{v})_{\mathbf{curl}}+(\mathbf{curl}\; \mathbf{u},\mathbf{curl}\; \mathbf{v})_{\mathbf{curl}}$ and the corresponding norm $\|\cdot\|_{\mathcal{U}}$.\\

Let $\pmb N$ be a $3\times3$ matrix valued function defined on $D$ such that $\pmb N\in L^\infty(D,\mathbb{R}^{3\times3})$.
{\bf Definition 2.1} A real matrix field $\pmb N$ is said to be bounded positive definite on $D$ if $\pmb N\in L^\infty(D,\mathbb{R}^{3\times3})$ and there exists a constant $\gamma>0$
such that
\begin{align*}
\bar{\mathbf{\xi}}\cdot \pmb N \mathbf{\xi}\geq\gamma|\mathbf{\xi}|^2,\forall \mathbf{\xi}\in \mathbb{C}^3 \text{~~a.e in } D.
\end{align*}\\
We further assume that $\pmb N,\pmb N^{-1}$ and $(\pmb N-\pmb I)^{-1}$ or $(\pmb I-\pmb N)^{-1}$ are bounded positive definite real matrix fields on $D$. The interior transmission eigenvalue problem
for the anisotropic Maxwell's equations in terms of electric fields is formulated as the problem of finding two vector valued functions $\mathbf{E}\in {L^2(D)}^3$  and $\mathbf{E_0}\in {L^2(D)}^3$ such that $\mathbf{E}-\mathbf{E_0}\in \mathcal{U}_0(D)$ satisfies \cite{CO,FE}:
\begin{align}
\mathbf{curl}\mathbf{curl}\mathbf{E}-k^2\pmb N\mathbf{E}=0    &~~~~~~~\text{in}~~~D,  \label{a2.1}\\
\mathbf{curl}\mathbf{curl}\mathbf{E_0}-k^2\mathbf{E_0}=0    &~~~~~~~\text{in}~~~D,    \label{a2.2}\\
\mathbf{E}\times\pmb\nu=\mathbf{E}_0\times\pmb\nu   &~~~~~~~ \text{on}~~~\partial D,  \label{a2.3}\\
\mathbf{curl}\mathbf{E}\times\pmb\nu=\mathbf{curl}\mathbf{E}_0\times\pmb\nu  &~~~~~~~ \text{on} ~~~\partial D. \label{a2.4}
\end{align}
\section{Dimension reduction scheme under spherical geometries}\label{Sec:Dimen}
 \indent We shall restrict our attention to the case where $D$ is a ball of radius $R$ and  $\pmb N=n\pmb I$  with  $n$ being a function along the radial direction.
In this case, by using vector spherical harmonic expansion we can reduce the problem to
 a sequence of equivalent one-dimensional TE and TM modes  that can be solved individually
 in parallel.
 \subsection{Vector spherical harmonics}
 There are several versions with different notation and properties for vector spherical harmonics (see e.g.,\cite{SAO, TTO, RBC,VSH}). We adopt in this paper the family of vector spherical harmonics in \cite{SAO,TTO}. \\

The spherical coordinates $(r,\theta,\phi)$ are related to the Cartesian coordinates $\mathbf{x}=(x_1,x_2,x_3)$ by
 \begin{align}
x_1=r\sin\theta \cos\phi, \quad x_2=r\sin\theta \sin\phi, \quad x_3=r\cos\theta.
\end{align}
Let
 \begin{align}
\mathbf{e}_r=\frac{\mathbf{x}}{r}, \quad \mathbf{e}_\theta=(\sin\theta \cos\phi,\sin\theta \sin\phi, -\sin\theta), \quad \mathbf{e}_\phi=(-\sin\phi,\cos\phi,0).
\end{align}
Then $\{\mathbf{e}_r,\mathbf{e}_\theta,\mathbf{e}_\phi\}$ forms a moving orthonormal  coordinate basis in $\mathbb{R}^3$.

Let $S$ be the unit spherical surface, and denote by $\Delta_S$ and $\nabla_S$ the Laplace-Beltrami and tangential gradient operators on $S$, namely,
\begin{align}
&\triangle_Su=\frac{1}{\sin\theta}\frac{\partial}{\partial\theta}(\sin\theta\frac{\partial u}{\partial\theta})+\frac{1}{\sin^2\theta}\frac{\partial^2u}{\partial\phi^2},\\
&\nabla_Su=\frac{\partial u}{\partial\theta}\mathbf{e}_\theta+\frac{1}{\sin\theta}\frac{\partial u}{\partial\phi}\mathbf{e}_\phi.
\end{align}
The spherical harmonics $\{Y_l^m\}$ (as normalized in \cite{SAO}) are eigenfunctions of $\Delta_S$, i.e.,
\begin{align}
\Delta_SY_l^m=-l(l+1)Y_l^m, \quad l\geq0,\; |m|\leq l;
\end{align}
and form an orthonormal basis for $L^2(S)$:
\begin{align}
\int_SY_l^mY_{l'}^{m'}dS=\delta_{ll'}\delta_{mm'}.
\end{align}
The family of  vector spherical harmonics is defined by \cite{SAO}:
\begin{align}
&\mathbf{T}_l^m=\nabla_SY_l^m\times\mathbf{e}_r=\frac{1}{\sin\theta}\frac{\partial Y_l^m}{\partial\phi}\mathbf{e}_\theta-\frac{\partial Y_l^m}{\partial\theta}\mathbf{e}_\phi, &&\text{for}~~ l\geq1, 0\leq|m|\leq l, \label{harmonic1}\\
&\mathbf{V}_l^m=(l+1)Y_l^m\mathbf{e}_r-\nabla_SY_l^m, &&\text{for}~~ l\geq0, 0\leq|m|\leq l,\label{harmonic2}\\
&\mathbf{W}_l^m=lY_l^m\mathbf{e}_r+\nabla_SY_l^m,&& \text{for}~~l\geq1, 0\leq|m|\leq l,\label{harmonic3}
\end{align}
where $\mathbf{V}_0^0=\frac{1}{\sqrt{4\pi}}\mathbf{e}_r$. With the understanding of $\mathbf{T}_0^0=\mathbf{W}_0^0=\mathbf{0},$ the indexes $\{l,m\}$ run over $\{(l,m):l\geq 0, 0\leq |m|\leq l\}.$

The following lemma was proved in \cite{SAO} ($(A.1)$ and $(A.2)$)\\
{\bf Lemma 3.1}
\begin{align}
&\int_S\mathbf{T}_l^m\cdot\overline{\mathbf{V}_{l'}^{m'}}dS=\int_S\mathbf{T}_l^m\cdot\overline{\mathbf{W}_{l'}^{m'}}dS=\int_S\mathbf{V}_l^m\cdot\overline{\mathbf{W}_{l'}^{m'}}dS=0,\\
&\int_S\mathbf{V}_l^m\cdot\overline{\mathbf{V}_{l'}^{m'}}dS=(l+1)(2l+1)\delta_{ll'}\delta_{mm'},\int_S\mathbf{T}_l^m\cdot\overline{\mathbf{T}_{l'}^{m'}}dS=l(l+1)\delta_{ll'}\delta_{mm'},\\
&\int_S\mathbf{W}_l^m\cdot\overline{\mathbf{W}_{l'}^{m'}}dS=l(2l+1)\delta_{ll'}\delta_{mm'}, \int_S\mathbf{T}_l^m\cdot\nabla_SY_l^mdS=0,\\
&\int_S\nabla_SY_l^m\cdot\nabla_SY_{l'}^{m'}dS=l(l+1)\delta_{ll'}\delta_{mm'}.
\end{align}
Let $f$ is a function along the radial direction and define the differentiation operators:
\begin{align}
&d_l^+=\frac{d}{dr}+\frac{l}{r}, \quad d_l^-=\frac{d}{dr}-\frac{l}{r}.
\end{align}
The following lemma is again proved in \cite{SAO} ($(A.9)$):\\
{\bf Lemma 3.2}
\begin{align}
&\mathbf{curl}(f\mathbf{V}_l^m)=(d_{l+2}^+f)\mathbf{T}_l^m, \quad \mathbf{curl}(f\mathbf{W}_l^m)=-(d_{l-1}^-f)\mathbf{T}_l^m,\\
&(2l+1)\mathbf{curl}(f\mathbf{T}_l^m)=(l+1)(d_{l+1}^+f)\mathbf{W}_l^m-l(d_l^-f)\mathbf{V}_l^m.
\end{align}

\subsection{Dimension reduction scheme}
Let us write
\begin{align}
&\mathbf{E}=\sum_{l=0}^{\infty}\sum_{|m|=0}^{l}(v_l^m(r)\mathbf{V}_l^m(\theta,\phi)+t_l^m(r)\mathbf{T}_l^m(\theta,\phi)+w_l^m(r)\mathbf{W}_l^m(\theta,\phi)),\label{expand1}\\
&\mathbf{E}_0=\sum_{l=0}^{\infty}\sum_{|m|=0}^{l}(\bar{v}_l^m(r)\mathbf{V}_l^m(\theta,\phi)+\bar{t}_l^m(r)\mathbf{T}_l^m(\theta,\phi)+\bar{w}_l^m(r)\mathbf{W}_l^m(\theta,\phi)). \label{expand2}
\end{align}
From Lemma 3.2 we can derive that
\begin{align}
&\mathbf{curl}\mathbf{curl}(v_l^m(r)\mathbf{V}_l^m(\theta,\phi))=\mathbf{curl}((d_{l+2}^+v_l^m(r))\mathbf{T}_l^m(\theta,\phi))\nonumber\\
&=\frac{l+1}{2l+1}(d_{l+1}^+d_{l+2}^+v_l^m(r))\mathbf{W}_l^m(\theta,\phi)-\frac{l}{2l+1}(d_{l}^-d_{l+2}^+v_l^m(r))\mathbf{V}_l^m(\theta,\phi), \label{eq1}\\
&\mathbf{curl}\mathbf{curl}(t_l^m(r)\mathbf{T}_l^m(\theta,\phi))=-(\frac{l}{2l+1}d_{l+2}^+d_{l}^-+\frac{l+1}{2l+1}d_{l-1}^-d_{l+1}^+)t_l^m(r)\mathbf{T}_l^m(\theta,\phi),\\
&\mathbf{curl}\mathbf{curl}(w_l^m(r)\mathbf{W}_l^m(\theta,\phi))=\frac{l}{2l+1}(d_{l}^-d_{l-1}^-w_l^m(r))\mathbf{V}_l^m(\theta,\phi)\nonumber\\
&-\frac{l+1}{2l+1}(d_{l+1}^+d_{l-1}^-w_l^m(r))\mathbf{W}_l^m(\theta,\phi).\label{eq3}
\end{align}
From \eqref{eq1}-\eqref{eq3} and the nonzero of eigenfunction, together with Lemma 3.1, \eqref{a2.1} can be reduced to
\begin{align}
&\frac{l}{2l+1}(d_{l}^-d_{l-1}^-w_l^m(r)-d_{l}^-d_{l+2}^+v_l^m(r))-k_l^2nv_l^m(r)=0,\\
&\frac{l+1}{2l+1}(d_{l+1}^+d_{l+2}^+v_l^m(r)-d_{l+1}^+d_{l-1}^-w_l^m(r))-k_l^2nw_l^m(r)=0,\\
&\frac{-1}{2l+1}(ld_{l+2}^+d_{l}^-+(l+1)d_{l-1}^-d_{l+1}^+)t_l^m(r)-k_l^2nt_l^m(r)=0,
\end{align}
where $l\geq1,0\leq|m|\leq l.$
Similarly, \eqref{a2.2} can be reduced to
\begin{align}
&\frac{l}{2l+1}(d_{l}^-d_{l-1}^-\bar{w}_l^m(r)-d_{l}^-d_{l+2}^+\bar{v}_l^m(r))-k_l^2\bar{v}_l^m(r)=0,\\
&\frac{l+1}{2l+1}(d_{l+1}^+d_{l+2}^+\bar{v}_l^m(r)-d_{l+1}^+d_{l-1}^-\bar{w}_l^m(r))-k_l^2\bar{w}_l^m(r)=0,\\
&\frac{-1}{2l+1}(ld_{l+2}^+d_{l}^-+(l+1)d_{l-1}^-d_{l+1}^+)\bar{t}_l^m(r)-k_l^2\bar{t}_l^m(r)=0,
\end{align}
where $l\geq1,0\leq|m|\leq l.$
From \eqref{harmonic1}-\eqref{harmonic3} and \eqref{expand1}, we can derive that
\begin{align}
&\mathbf{E}\times\pmb\nu=\mathbf{E}\times\mathbf{e}_r=\sum_{l=0}^{\infty}\sum_{|m|=0}^{l}((w_l^m(r)-v_l^m(r))\mathbf{T}_l^m-t_l^m(r)\nabla_SY_l^m).
\end{align}
Similarly, we can derive that
\begin{align}
&\mathbf{E}_0\times\pmb\nu=\mathbf{E}_0\times\mathbf{e}_r=\sum_{l=0}^{\infty}\sum_{|m|=0}^{l}((\bar{w}_l^m(r)-\bar{v}_l^m(r))\mathbf{T}_l^m-\bar{t}_l^m(r)\nabla_SY_l^m).
\end{align}
From Lemma 3.1, the boundary condition \eqref{a2.3} can be reduced to
\begin{align}
&w_l^m(R)-v_l^m(R)=\bar{w}_l^m(R)-\bar{v}_l^m(R), t_l^m(R)=\bar{t}_l^m(R).
\end{align}
From Lemma 3.2 and \eqref{harmonic1}-\eqref{harmonic3} we can derive that
\begin{eqnarray}
&& \mathbf{curl}\; \mathbf{E}\times\pmb\nu = \mathbf{curl}\; \mathbf{E}\times\mathbf{e}_r \\
&=& \sum_{l=0}^{\infty} \sum_{|m|=0}^1 ((d_{l-1}^-w_l^m(r)-d_{l+2}^+v_l^m(r)) \nabla_S Y_l^m
+ (\frac{l}{2l+1}d_{l}^- + \frac{l+1}{2l+1} d_{l+1}^+) t_l^m(r) \mathbf{T}_l^m). \nonumber
\end{eqnarray}
Similarly,  we can derive that
\begin{eqnarray}
& & \mathbf{curl}\; \mathbf{E}_0\times\pmb\nu = \mathbf{curl}\; \mathbf{E}_0\times\mathbf{e}_r \\
&=& \sum_{l=0}^{\infty} \sum_{|m|=0}^{l} ((d_{l-1}^-\bar{w}_l^m(r) - d_{l+2}^ + \bar{v}_l^m(r)) \nabla_S Y_l^m
+ (\frac{l}{2l+1}d_{l}^- + \frac{l+1}{2l+1}d_{l+1}^+) \bar{t}_l^m(r) \mathbf{T}_l^m ). \nonumber
\end{eqnarray}
From Lemma 3.1, the boundary condition \eqref{a2.4} can be reduced to
\begin{align}
&d_{l-1}^-w_l^m(R)-d_{l+2}^+v_l^m(R)=d_{l-1}^-\bar{w}_l^m(R)-d_{l+2}^+\bar{v}_l^m(R),\\
&(\frac{l}{2l+1}d_{l}^-+\frac{l+1}{2l+1}d_{l+1}^+)t_l^m(R)=(\frac{l}{2l+1}d_{l}^-+\frac{l+1}{2l+1}d_{l+1}^+)\bar{t}_l^m(R).
\end{align}
Note that the modes $t_l^m$ (coefficients of $\mathbf{T}_l^m$) are decoupled from the modes $v_l^m$ and $w_l^m$.

In summary, we only have to solve the following sequence ($l\geq1$ and $|m|\leq l$) of one-dimensional eigenvalue problems, i.e., the so-called TE mode:
\begin{align}
&\frac{-1}{2l+1}(ld_{l+2}^+d_{l}^-+(l+1)d_{l-1}^-d_{l+1}^+)t_l^m(r)-k_l^2nt_l^m(r)=0,&r\in(0,R), \label{TE1}\\
&\frac{-1}{2l+1}(ld_{l+2}^+d_{l}^-+(l+1)d_{l-1}^-d_{l+1}^+)\bar{t}_l^m(r)-k_l^2\bar{t}_l^m(r)=0,&r\in(0,R),\label{TE2}\\
&t_l^m(R)=\bar{t}_l^m(R),\label{TE3}\\
&(\frac{l}{2l+1}d_{l}^-+\frac{l+1}{2l+1}d_{l+1}^+)t_l^m(R)=(\frac{l}{2l+1}d_{l}^-+\frac{l+1}{2l+1}d_{l+1}^+)\bar{t}_l^m(R),\label{TE4}
\end{align}
and TM mode:
\begin{align}
&\frac{l}{2l+1}(d_{l}^-d_{l-1}^-w_l^m(r)-d_{l}^-d_{l+2}^+v_l^m(r))-k_l^2nv_l^m(r)=0,&r\in(0,R),\label{TM1}\\
&\frac{l+1}{2l+1}(d_{l+1}^+d_{l+2}^+v_l^m(r)-d_{l+1}^+d_{l-1}^-w_l^m(r))-k_l^2nw_l^m(r)=0,&r\in(0,R),\label{TM2}\\
&\frac{l}{2l+1}(d_{l}^-d_{l-1}^-\bar{w}_l^m(r)-d_{l}^-d_{l+2}^+\bar{v}_l^m(r))-k_l^2\bar{v}_l^m(r)=0,&r\in(0,R),\label{TM3}\\
&\frac{l+1}{2l+1}(d_{l+1}^+d_{l+2}^+\bar{v}_l^m(r)-d_{l+1}^+d_{l-1}^-\bar{w}_l^m(r))-k_l^2\bar{w}_l^m(r)=0,&r\in(0,R),\label{TM4}\\
&w_l^m(R)-v_l^m(R)=\bar{w}_l^m(R)-\bar{v}_l^m(R),\label{TM5}\\
&d_{l-1}^-w_l^m(R)-d_{l+2}^+v_l^m(R)=d_{l-1}^-\bar{w}_l^m(R)-d_{l+2}^+\bar{v}_l^m(R).\label{TM6}
\end{align}

\section{Weak formulation and error estimation of the TE mode}
For brief, we shall only give the error analysis in detail for the TE mode.  For the TM mode, it is a coupled system with four unknown functions which is a
challenging problem for numerical calculation. We propose an efficient numerical algorithm by using Legendre approximation based on vector basis functions.
\subsection{Weak formulation and discrete formulation}
By simplification, the TE mode can be rewritten as follows:
\begin{align}
&-\frac{1}{r^2}\partial_r(r^2\partial_rt_l^m(r))+\frac{l(l+1)}{r^2}t_l^m(r)-k_l^2nt_l^m(r)=0,&r\in(0,R), \label{TE11}\\
&-\frac{1}{r^2}\partial_r(r^2\partial_r\bar{t}_l^m(r))+\frac{l(l+1)}{r^2}\bar{t}_l^m(r)-k_l^2\bar{t}_l^m(r)=0,&r\in(0,R),\label{TE21}\\
&t_l^m(R)=\bar{t}_l^m(R),\label{TE31}\\
&\partial_rt_l^m(R)=\partial_r\bar{t}_l^m(R).\label{TE41}
\end{align}
Next, we formulate \eqref{TE11} - \eqref{TE41}  as a equivalent fourth-order eigenvalue problem.
Let $\tilde{u}_l=t_l^m(r)-\bar{t}_l^m(r)$, $\tilde{L}_l\tilde{u}_l=\frac{1}{r^2}\partial_r(r^2\partial_r\tilde{u}_l)-\frac{l(l+1)}{r^2}\tilde{u}_l$.
Subtracting \eqref{TE21} from \eqref{TE11}, we obtain
\begin{align}
&-(\tilde{L}_l\tilde{u}_l+k_l^2\tilde{u}_l)=k_l^2(n-1)t_l^m
\end{align}
Divding $n-1$ and applying $\tilde{L}_l+k_l^2n$ to both sides of above equation, we obtain
\begin{align}
&(\tilde{L}_l+k_l^2n)\frac{1}{n-1}(\tilde{L}_l+k_l^2)\tilde{u}_l=0.\label{four}
\end{align}
Then the \eqref{TE11}-\eqref{TE41} is equivalent to the following fourth order eigenvalue problem:
\begin{align}
&(\tilde{L}_l+k_l^2n)\frac{1}{n-1}(\tilde{L}_l+k_l^2)\tilde{u}_l=0.\label{four1}\\
&\tilde{u}_l(R)=\tilde{u}_l'(R)=0.\label{four2}
\end{align}

Let $r=\frac{t+1}{2}R$, $\tilde{n}(t)=n(\frac{t+1}{2}R)$, $u_l=\tilde{u}_l(\frac{t+1}{2}R),$ $L_lu_l=\frac{1}{(t+1)^2}\partial_t((t+1)^2\partial_tu_l)-\frac{l(l+1)}{(t+1)^2}u_l$.
Then \eqref{four1}-\eqref{four2}  can be restated as follows:
\begin{align}
&(\frac{4}{R^2}L_l+k_l^2\tilde{n})\frac{1}{\tilde{n}-1}(\frac{4}{R^2}L_l+k_l^2)u_l=0, ~~\text{in}~~(-1,1), \label{four3}\\
&u_l(1)=u_l'(1)=0.\label{four4}
\end{align}
 We now define the usual weighted Sobolev space:
 \begin{align}
L_{\omega}^2(I):=\{u:\int_I\omega u^2dt<\infty\}
\end{align}
equipped with the inner product and norm
 \begin{align}
(u,v)_\omega=\int_I\omega uvdt, \quad  \|u\|_w=(\int_I\omega u^2dt)^{\frac{1}{2}},
\end{align}
where $I=(-1,1)$ and $\omega=1+t$ is a weight function.
Next, we introduce the following non-uniformly weighted Sobolev space:
\begin{align}
&H_{0,\omega,l}^1(I):=\{u:\partial_t^ku\in L_{\omega^{2k}}^2(I),k=0,1,u(1)=0\};\\
&H_{0,\omega,l}^2(I):=\{u:\partial_t^ku\in L_{\omega^{2k-2}}^2(I),k=0,1,2,u(\pm1)=u'(1)=0\}
\end{align}
equipped with the corresponding inner product and norm
 \begin{align}
&(u,v)_{1,\omega,l}=\sum_{k=0}^1(\partial_t^ku,\partial_t^kv)_{\omega^{2k}}, \|u\|_{1,\omega,l}=(u,u)_{1,\omega,l}^{\frac{1}{2}};\\
&(u,v)_{2,\omega,l}=\sum_{k=0}^2(\partial_t^ku,\partial_t^kv)_{\omega^{2k-2}}, \|u\|_{2,\omega,l}=(u,u)_{2,\omega,l}^{\frac{1}{2}}.
\end{align}
Then the weak form of \eqref{four3} -\eqref{four4} is: Find $(k^2_l,u_l)\in \mathbb{C}\times H_{0,\omega,l}^2(I)$, such that
\begin{align}
\int_{-1}^{1}\frac{1}{\tilde{n}-1}(t+1)^2(L_l+\frac{R^2}{4}k_l^2)u_l(L_l+\frac{R^2}{4}k_l^2\tilde{n})\bar{v}dt=0, \forall v\in  H_{0,\omega,l}^2(I). \label{weak1}
\end{align}
Note the condition $u_l(-1)=0$ in $ H_{0,\omega,l}^2(I)$ is a essential polar condition which should be imposed for the well-posedness of the weak form \eqref{weak1} (and the same type of pole condition for $v$).

We now introduce an associated generalized eigenvalue problems. First we define
\begin{align}
&A_{\tau_l}(u_l,v):=(\frac{1}{\tilde{n}-1}(L_l+\frac{R^2}{4}\tau_l)u_l,(L_l+\frac{R^2}{4}\tau_l)v)_{\omega^2}+\tau_l^2(\frac{R^2}{4})^2(u_l,v)_{\omega^2}
\end{align}
for $\tilde{n}>1$, and
\begin{align}
\tilde{A}_{\tau_l}(u_l,v):&=(\frac{1}{1-\tilde{n}}(L_l+\frac{R^2}{4}\tau_l\tilde{n})u_l,(L_l+\frac{R^2}{4}\tau_l\tilde{n})v)_{\omega^2}+\tau_l^2(\frac{R^2}{4})^2(\tilde{n}u_l,v)_{\omega^2}\nonumber\\
&=(\frac{\tilde{n}}{1-\tilde{n}}(L_l+\frac{R^2}{4}\tau_l)u_l,(L_l+\frac{R^2}{4}\tau_l)v)_{\omega^2}+(L_lu_l,L_lv)_{\omega^2}
\end{align}
for $\tilde{n}<1$,
\begin{align}
&B(u_l,v):=\frac{R^2}{4}((\partial_tu_l,\partial_tv)_{\omega^2}+l(l+1)(u_l,v)),
\end{align}
where $\tau_l=k_l^2$.
Then \eqref{weak1} can be written as either
\begin{align}
&A_{\tau_l}(u_l,v)-\tau_lB(u_l,v)=0, \forall v\in H_{0,\omega,l}^2(I), \label{gener1}
\end{align}
or
\begin{align}
&\tilde{A}_{\tau_l}(u_l,v)-\tau_lB(u_l,v)=0, \forall v\in H_{0,\omega,l}^2(I).\label{gener2}
\end{align}
The associated generalized eigenvalue problem is: Find $(\lambda(\tau_l),u_l)\in \mathbb{C}\times H_{0,\omega,l}^2(I)$, such that
\begin{align}
&A_{\tau_l}(u_l,v)-\lambda(\tau_l)B(u_l,v), \forall v\in H_{0,\omega,l}^2(I),\label{weak2}
\end{align}
or
\begin{align}
&\tilde{A}_{\tau_l}(u_l,v)-\lambda(\tau_l)B(u_l,v)=0, \forall v\in H_{0,\omega,l}^2(I), \label{weak3}
\end{align}
where $\lambda(\tau_l)$ is a continuous function of $\tau_l$. From \eqref{gener1}-\eqref{gener2}, we know that a transmission eigenvalue is the root of $f(\tau_l)=\lambda(\tau_l)-\tau_l$.

Let $P_N$ be the space of polynomials of degree less than or equal to $N$, and setting $X_N=P_N\cap H_{0,\omega,l}^2$.
Then the discrete formulation of \eqref{weak2} is: Find $(\lambda_{N}(\tau_l),u_{lN})\in \mathbb{C}\times X_N$, such that
\begin{align}
A_{\tau_l}(u_{lN},v_N)-\lambda_N(\tau_l)B(u_{lN},v_N), \forall v_N\in X_N.\label{discrete1}
\end{align}
The discrete formulation of \eqref{weak3} is: Find $(\lambda_{N}(\tau_l),u_{lN})\in \mathbb{C}\times X_N$, such that
\begin{align}
\tilde{A}_{\tau_l}(u_{lN},v_N)-\lambda_N(\tau_l)B(u_{lN},v_N), \forall v_N\in X_N.\label{discrete2}
\end{align}
\subsection{Error estimation of approximate eigenvalues}
Hereafter, we shall use the expression $a\lesssim b$ to mean that there exists a positive constant $C$ such that $a\leq Cb$.\\
~~~~~\\
{\bf Lemma 4.1}
It holds:
\begin{align*}
&\|u\|_{2,\omega,l}^2\lesssim\|L_lu\|_{\omega^2}^2\lesssim\|u\|_{2,\omega,l}^2.
\end{align*}
\begin{proof}
Since
\begin{align*}
&\|L_lu\|_{\omega^2}^2=\int_{-1}^1(L_lu)^2(t+1)^2dt\\
&=\int_{-1}^1\frac{1}{(t+1)^2}(\partial_t((1+t)^2\partial_tu)-l(l+1)u)^2dt\\
&=\int_{-1}^1((t+1)\partial_t^2u+2\partial_tu-\frac{l(l+1)}{1+t}u)^2dt\\
&=\int_{-1}^1((t+1)^2(\partial_t^2u)^2+2(l^2+l+1)(\partial_tu)^2+l(l^2-1)(l+2)\frac{1}{(1+t)^2}u^2)dt,
\end{align*}
then
\begin{align*}
&\|L_lu\|_{\omega^2}^2\lesssim\int_{-1}^1((t+1)^2(\partial_t^2u)^2+(\partial_tu)^2+\frac{1}{(1+t)^2}u^2)dt=\|u\|_{2,\omega,l}^2.
\end{align*}
From Hardy inequality (cf. B 8.6 in \cite{shen2011}) we have
\begin{align*}
&\int_{-1}^1\frac{1}{(t+1)^2}u^2dt\lesssim \int_{-1}^1(\partial_tu)^2dt.
\end{align*}
Thus, we can obtain
\begin{align*}
&\|L_lu\|_{\omega^2}^2\geq\int_{-1}^1((t+1)^2(\partial_t^2u)^2+2(l^2+l+1)(\partial_tu)^2dt\\
&\gtrsim\int_{-1}^1((t+1)^2(\partial_t^2u)^2+2(l^2+l)(\partial_tu)^2+\frac{1}{(1+t)^2}u^2)dt\\
&\gtrsim\|u\|_{2,\omega,l}^2.
\end{align*}
\end{proof}
\begin{theorem}\label{theorem1}
Let $\tilde{n}\in L^\infty(I)$ satisfy
 \begin{align*}
&1+\alpha\leq \tilde{n}_*\leq\tilde{n}\leq\tilde{n}^*<\infty~~
\text{or}~~0<\tilde{n}_*\leq\tilde{n}\leq\tilde{n}^*<1-\beta
\end{align*}
for some $\alpha>0$ and $\beta>0$ positive constants. Then $A_{\tau_l}$ or $\tilde{A}_{\tau_l}$ is a continuous and coercive sesquilinear form on $H_{0,\omega,l}^2(I)\times H_{0,\omega,l}^2(I)$, i.e.,
 \begin{align}
 &A_{\tau_l}(u,u) ~ \text{or}~ \tilde{A}_{\tau_l}(u,u)\gtrsim\|u\|_{2,\omega,l}^2, \label{posi1}\\
&|A_{\tau_l}(u,v)| ~ \text{or}~ |\tilde{A}_{\tau_l}(u,v)|\lesssim\|u\|_{2,\omega,l}\|v\|_{2,\omega,l},\label{posi2}
\end{align}
where $\tilde{n}_*=\inf_{I}(\tilde{n})$ and $\tilde{n}^*=\sup_{I}(\tilde{n})$.
\end{theorem}
\begin{proof}
We shall only give the proof in the case of $1+\alpha\leq \tilde{n}_*\leq\tilde{n}\leq\tilde{n}^*<\infty$. It can be similarly derived for the case of $~~0<\tilde{n}_*\leq\tilde{n}\leq\tilde{n}^*<1-\beta$.
\end{proof}

Since
 \begin{align*}
A_{\tau_l}(u,u)&=(\frac{1}{\tilde{n}-1}(L_l+\frac{R^2}{4}\tau_l)u,(L_l+\frac{R^2}{4}\tau_l)u)_{\omega^2}+(\frac{R^2}{4}\tau_l)^2(u,u)_{\omega^2}\\
&\geq \frac{1}{\tilde{n}^*-1}((L_l+\frac{R^2}{4}\tau_l)u,(L_l+\frac{R^2}{4}\tau_l)u)_{\omega^2}+(\frac{R^2}{4}\tau_l)^2(u,u)_{\omega^2}\\
&\geq\frac{1}{\tilde{n}^*-1}(\|L_lu\|_{\omega^2}^2-\frac{R^2}{2}\tau_l\|L_lu\|_{\omega^2}\|u\|_{\omega^2}+(\frac{R^2}{4}\tau_l)^2\|u\|_{\omega^2}^2)+(\frac{R^2}{4}\tau_l)^2\|u\|_{\omega^2}^2\\
&=\varepsilon(\frac{R^2}{4}\tau_l\|u\|_{\omega^2}-\frac{1}{\varepsilon(\tilde{n}^*-1)}\|L_lu\|_{\omega^2})^2+(\frac{1}{\tilde{n}^*-1}-\frac{1}{\varepsilon(\tilde{n}^*-1)^2})\|L_lu\|_{\omega^2}^2\\
&+(1+\frac{1}{\tilde{n}^*-1}-\varepsilon)(\frac{R^2}{4}\tau_l)^2\|u\|_{\omega^2}^2\geq (\frac{1}{\tilde{n}^*-1}-\frac{1}{\varepsilon(\tilde{n}^*-1)^2})\|L_lu\|_{\omega^2}^2
\end{align*}
for $\frac{1}{\tilde{n}^*-1}<\varepsilon<\frac{1}{\tilde{n}^*-1}+1$,  then from Lemma 4.1 we can obtain \eqref{posi1}.
 Due to
 \begin{align*}
|A_{\tau_l}(u,v)|&=|(\frac{1}{\tilde{n}-1}(L_l+\frac{R^2}{4}\tau_l)u,(L_l+\frac{R^2}{4}\tau_l)v)_{\omega^2}+(\frac{R^2}{4}\tau_l)^2(u,v)_{\omega^2}|\\
&=|(\frac{1}{\tilde{n}-1}L_lu,L_lv)_{\omega^2}+\frac{R^2}{4}\tau_l(\frac{1}{\tilde{n}-1}L_lu,v)_{\omega^2}+\frac{R^2}{4}\tau_l(\frac{1}{\tilde{n}-1}u,L_lv)_{\omega^2}\\
&+(\frac{R^2}{4}\tau_l)^2(\frac{1}{\tilde{n}-1}u,v)_{\omega^2}+(\frac{R^2}{4}\tau_l)^2(u,v)_{\omega^2}|\\
&\leq\frac{1}{\alpha}\|L_lu\|_{\omega^2}\|L_lv\|_{\omega^2}+\frac{R^2}{4\alpha}\tau_l\|L_lu\|_{\omega^2}\|v\|_{\omega^2}+\frac{R^2}{4\alpha}\tau_l\|u\|_{\omega^2}\|L_lv\|_{\omega^2}\\
&+\frac{1}{\alpha}(\frac{R^2}{4}\tau_l)^2\|u\|_{\omega^2}\|v\|_{\omega^2}+(\frac{R^2}{4}\tau_l)^2\|u\|_{\omega^2}\|v\|_{\omega^2}\\
&\leq(\frac{1}{\alpha}+\frac{R^2}{2\alpha}\tau_l+(\frac{R^2}{4}\tau_l)^2(\frac{1}{\alpha}+1))\|L_lu\|_{\omega^2}\|L_lv\|_{\omega^2},
\end{align*}
then from Lemma 4.1 we can obtain \eqref{posi2}. \hspace*{1cm}  $\Box$

Similar to the proof of Theorem 1, we can derive the following Theorem:
\begin{theorem}
$B(u,v)$ is a continuous and coercive bilinear form on $H_{0,\omega,l}^1(I) \times H_{0,\omega,l}^1(I)$, i.e.,
 \begin{align*}
&|B(u,v)|\lesssim\|u\|_{1,\omega,l}\|v\|_{1,\omega,l},\\
&B(u,u)\gtrsim \|u\|_{1,\omega,l}^2.
\end{align*}
\end{theorem}

To give the error analysis, we will use extensively the minimax principle.\\
\noindent {\bf Lemma 4.2} {\it Let $\lambda^{m}(\tau_l)$ denote the $m$-th eigenvalues of (\ref{weak2}) and $V_{m}$ be any $m$-dimensional
subspace of $H_{0,\omega,l}^2(I) $. Then, for $\lambda^1(\tau_l)\leq \lambda^2(\tau_l) \leq \cdots \leq \lambda^m(\tau_l) \leq \cdots $, there holds}
\begin{eqnarray}\label{e3.3}
\lambda^m(\tau_l)=\min_{V_m\subset H_{0,\omega,l}^2(I) } \max_{v\in V_m}\frac{A_{\tau_l}(v,v)}{B(v,v)}.
\end{eqnarray}
{\bf Proof.} See Theorem 3.1 in \cite{mg2003}.\hspace*{1cm}   $\Box$

\noindent {\bf Lemma 4.3} {\it Let $\lambda^{m}(\tau_l)$ denote the $m$-th eigenvalues of (\ref{weak2}) and
be arranged in an ascending order, and define
\begin{eqnarray*}
E_{i,j}= \hbox{\rm span}\left\{u_l^i,\cdots,u_l^j\right\},
\end{eqnarray*}
where $u_l^i$ is the eigenfunction corresponding to the eigenvalue $\lambda^i(\tau_l)$. Then we have}
\begin{eqnarray}\label{e3.4}
\lambda^m(\tau_l)= \max_{v\in E_{k,m}}\frac{A_{\tau_l}(v,v)}{B(v,v)} && k\leq m,  \\\label{e3.5}
\lambda^m(\tau_l)= \min_{v\in E_{m,n}}\frac{A_{\tau_l}(v,v)}{B(v,v)} && m\leq n .
\end{eqnarray}
{\bf Proof.} See Lemma 3.2 in \cite{mg2003}. \hspace*{1cm}  $\Box$

It is true that the minimax principle is also valid for the discrete formulation (\ref{discrete1}) (see \cite{mg2003}).

\noindent {\bf Lemma 4.4} {\it Let $\lambda_{N}^m (\tau_l)$ denote the $m$-th  eigenvalues of (\ref{discrete1}), and $V_m$ be any $m$-dimensional subspace of $X_N$.
Then, for $\lambda_{N}^1 (\tau_l)\leq \lambda_{N}^2 (\tau_l)\leq \cdots \leq\lambda_{N}^m (\tau_l)\leq \cdots$, there holds}
\begin{eqnarray}\label{e3.6}
\lambda_{N}^m (\tau_l)=\min_{V_m\subset X_N} \max_{v\in V_m}\frac{A_{\tau_l}(v,v)}{B(v,v)}.
\end{eqnarray}
Let $\prod_{N}^{2,l}:H_{0,\omega,l}^2(I)\rightarrow X_N$ be an orthogonal projection, defined by
 \begin{align*}
&A_{\tau_l}(u_l-\Pi_{N}^{2,l}u_l,v)=0,\forall v\in X_N.
\end{align*}
\begin{theorem}
 {\it Let $\lambda_{N}^m (\tau_l)$ be obtained by solving (\ref{discrete1}) and be an approximation of $\lambda^m (\tau_l)$, an eigenvalue of (\ref{weak2}). Then, we have}
\begin{eqnarray}\label{e3.7}
0<\lambda^m (\tau_l)\leq\lambda_{N}^m (\tau_l)\leq\lambda^m (\tau_l)\max_{v\in E_{1,m}}\frac{B(v,v)}{B(\Pi_{N}^{2,l}v,\Pi_{N}^{2,l}v)}.
\end{eqnarray}
\end{theorem}
\begin{proof} According to the coerciveness of $A_{\tau_l}(u,v)$ and $B(u,v)$ we can get $\lambda^m (\tau_l)>0$. From $X_N \subset H_{0,\omega,l}^2(I)$, together with (\ref{e3.3}) and (\ref{e3.6}) we can obtain $\lambda^m (\tau_l)\leq\lambda_{N}^m (\tau_l)$.
Let $\Pi_{N}^{2,l}E_{1,m}$ denote the space spanned by $\Pi_{N}^{2,l}u_l^1,\Pi_{N}^{2,l}u_l^2,\cdots,\Pi_{N}^{2,l}u_l^m$.  It is obvious that
$\Pi_{N}^{2,l}E_{1,m}$ is a $m$-dimensional subspace of $X_N$. From the minimax principle, we have
\begin{eqnarray*}
\lambda_{N}^m(\tau_l)\leq  \max_{v\in\Pi_{N}^{2,l}E_{1,m}}\frac{A_{\tau_l}(v,v)}{B(v,v)}=\max_{v\in E_{1,m}}\frac{A_{\tau_l}(\Pi_{N}^{2,l}v,\Pi_{N}^{2,l}v)}{B(\Pi_{N}^{2,l}v,\Pi_{N}^{2,l}v)}.
\end{eqnarray*}
Since
$
A_{\tau_l}(v,v)=A_{\tau_l}(\Pi_{N}^{2,l}v,\Pi_{N}^{2,l}v)+2A_{\tau_l}(v-\Pi_{N}^{2,l}v,\Pi_{N}^{2,l}v)+A_{\tau_l}(v-\Pi_{N}^{2,l}v,v-\Pi_{N}^{2,l}v),
$
from  $A_{\tau_l}(v-\Pi_{N}^{2,l}v,\Pi_{N}^{2,l}v)=0$ and the non-negativity of $A_{\tau_l}(v-\Pi_{N}^{2,l}v,v-\Pi_{N}^{2,l}v)$, we have
\begin{eqnarray*}
A_{\tau_l}(\Pi_{N}^{2,l}v,\Pi_{N}^{2,l}v)\leq A_{\tau_l}(v,v).
\end{eqnarray*}
Thus, we have
\begin{eqnarray*}
\lambda_{N}^m(\tau_l) &\leq& \max_{v\in E_{1,m}}\frac{A_{\tau_l}(v,v)}{B(\Pi_{N}^{2,l}v,\Pi_{N}^{2,l}v)}\\
&=&\max_{v\in E_{1,m}}\frac{A_{\tau_l}(v,v)}{B(v,v)}\frac{B(v,v)}{B(\Pi_{N}^{2,l}v,\Pi_{N}^{2,l}v)}\\
&\leq&\lambda^m(\tau_l)\max_{v\in E_{1,m}}\frac{B(v,v)}{B(\Pi_{N}^{2,l}v,\Pi_{N}^{2,l}v)}.
\end{eqnarray*}
The proof of Theorem 3 is complete. \hspace*{1cm}   $\Box$\\
\end{proof}

We denote the Jacobi weight function of index $(\alpha,\beta)$ by
$
\omega^{\alpha,\beta}(x)=(1-t)^\alpha(1+t)^\beta
$
 and  introduce the following non-uniformly weighted Sobole spaces:
\begin{align*}
H_{\omega^{\alpha,\beta},*}^s(I):=\{u:\partial_t^ku\in L_{\omega^{\alpha+k,\beta+k}},0\leq k\leq s\},
\end{align*}
equipped with the inner product and norm
\begin{align*}
(u,v)_{s,\omega^{\alpha,\beta},*}=\sum_{k=0}^{s}(\partial_t^ku,\partial_t^kv)_{\omega^{\alpha+k,\beta+k}},\|u\|_{s,\omega^{\alpha,\beta},*}=(u,u)_{s,\omega^{\alpha,\beta},*}^{\frac{1}{2}},
\end{align*}
and
\begin{align*}
H_{\omega^{-2,-2},l}^s(I):=\{u\in H_{0,\omega,l}^2(I)\cap H^2(I):\partial_t^ku\in L_{\omega^{-2+k,-2+k}},3\leq k\leq s\},
\end{align*}
equipped with the inner product and norm
\begin{align*}
&(u,v)_{s,\omega^{-2,-2},l}=(u,v)_{2,\omega,l}+\sum_{k=3}^{s}(\partial_t^ku,\partial_t^kv)_{\omega^{-2+k,-2+k}},\\
&\|u\|_{s,\omega^{-2,-2},l}=(u,u)_{s,\omega^{-2,-2},l}^{\frac{1}{2}}.
\end{align*}
Define the orthogonal projection $\pi_{N,\omega^{-2,-2}}:L_{\omega^{-2,-2}}^2(I)\rightarrow Q_N^{-2,-2}$ by
\begin{align}
(u-\pi_{N,\omega^{-2,-2}}u,v_N)_{\omega^{-2,-2}}=0, \forall v_N\in Q_N^{-2,-2}, \label{project}
\end{align}
where $Q_N^{-2,-2}=\{\phi\in P_N:\phi(\pm1)=\phi'(\pm1)=0\}.$
From the Theorem 1.8.2 in \cite{shen2006} we have the following Lemma:\\
{\bf Lemma 4.5.}  For any $u\in H_{\omega^{-2,-2},*}^s(I)$, the following inequality holds:
\begin{align*}
\|\partial_t^2(\pi_{N,\omega^{-2,-2}}u-u)\|\lesssim N^{2-s}\|\partial_t^su\|_{\omega^{-2+s,-2+s}}.
\end{align*}

\begin{theorem}
There exists an operator $\pi_N^{2,l}:H_{0,\omega,l}^2(I)\rightarrow P_N^{0,l}$ such that $\pi_N^{2,l}u(\pm1)=u(\pm1)=0$, and
$\partial_t\pi_N^{2,l}u(-1)=\partial_tu(-1)$, $\partial_t\pi_N^{2,l}u(1)=\partial_tu(1)=0$ and for $u\in H_{\omega^{-2,-2},l}^s(I)$ with $s\geq 2$, there holds
\begin{align*}
&\|\partial_t^2(\pi_{N}^{2,l}u-u)\|\lesssim N^{2-s}(\|\partial_t^su\|_{\omega^{-2+s,-2+s}}+\|\partial_t^2u\|) (s\leq3),\\
&\|\partial_t^2(\pi_{N}^{2,l}u-u)\|\lesssim N^{2-s}\|\partial_t^su\|_{\omega^{-2+s,-2+s}} (s>3),
\end{align*}
where $P_N^{0,l}=\{\phi\in P_N:\phi(\pm1)=\phi'(1)=0\}$.
\end{theorem}
\begin{proof}
Let $u_*(t)=\frac{1}{4}(t+1)(1-t)^2\partial_tu(-1)$ for $\forall u\in H_{0,\omega,l}^2(I)$. By construction, we have $\partial_t^k(u-u_*)(\pm1)=0,(k=0,1).$
If $u\in H_{\omega^{-2,-2},l}^s(I)$ , then we have  $u-u_*\in H_{\omega^{-2,-2},*}^s(I)$. In fact, from Hardy inequality (cf. B 8.8 in \cite{shen2011}) we have
\begin{align*}
&\int_I\omega^{-2,-2}(u-u_*)^2dt\lesssim\int_I(\partial_t(u-u_*))^2dt,\\
&\int_I\omega^{-2,-2}(\partial_t(u-u_*))^2dt\lesssim\int_I(\partial_t^2(u-u_*))^2dt.
\end{align*}
Thus, we can derive that
\begin{align*}
\int_I\omega^{-2,-2}(u-u_*)^2dt&\lesssim\int_I\omega^{-1,-1}(\partial_t(u-u_*))^2dt\lesssim\int_I(\partial_t^2(u-u_*))^2dt,
\end{align*}
Since
\begin{align*}
\int_I(\partial_t^2u_*)^2dt&= \int_I((\frac{3}{2}t-\frac{1}{2})\partial_tu(-1))^2dt\\
&=2(\partial_tu(-1))^2=2(\int_I\partial_t^2udt)^2\leq4\int_I(\partial^2_tu)^2dt,
\end{align*}
then we have
\begin{align}
\int_I(\partial_t^2(u-u_*))^2dt&\lesssim\int_I(\partial_t^2u)^2dt+\int_I(\partial_t^2u_*)^2dt \lesssim \int_I(\partial_t^2u)^2dt.\label{inequa1}
\end{align}
Similarly, we can derive that $
\int_I\omega^{1,1}(\partial_t^3u_*)^2dt=3(\partial_tu(-1))^2\lesssim\int_I(\partial^2_tu)^2dt.
$
Thus we have
\begin{align}
\int_I\omega^{1,1}(\partial_t^3(u-u_*))^2dt&\lesssim\int_I\omega^{1,1}(\partial_t^3u)^2dt+\int_I\omega^{1,1}(\partial_t^3u_*)^2dt \nonumber \\
&\lesssim \int_I\omega^{1,1}(\partial_t^3u)^2dt+\int_I(\partial^2_tu)^2dt.\label{inequa2}
\end{align}
For $k>3$, we have
\begin{align}
\int_I\omega^{-2+k,-2+k}(\partial_t^k(u-u_*))^2dt&=\int_I\omega^{-2+k,-2+k}(\partial_t^ku)^2dt. \label{inequa3}
\end{align}
Thus, $u-u_*\in H_{\omega^{-2,-2},*}^s(I)$ and we can define
\begin{align*}
\pi_N^{2,l}u=\pi_{N,\omega^{-2,-2}}(u-u_*)+u_* \in P_N^{0,l}, \forall  u\in H_{\omega^{-2,-2},l}^s(I)
\end{align*}
Then by Lemma 4.5 we can obtain
\begin{align*}
\|\partial_t^2(\pi_N^{2,l}u-u)\|=\|\partial_t^2\pi_{N,\omega^{-2,-2}}(u-u_*)-(u-u_*)\|\lesssim N^{2-s}\|\partial_t^s(u-u_*)\|_{\omega^{-2+s,-2+s}}.
\end{align*}
Together with \eqref{inequa1}, \eqref{inequa2} and \eqref{inequa3} we can get desired results.  $\Box$\\
\end{proof}
\begin{theorem}
Let $\lambda^m_{N}(\tau_l)$ is the $m$-th approximate eigenvalue of $\lambda^m(\tau_l)$.\\
 If $\{u_l^{i}\}_{i=1}^{m} \subset H_{\omega^{-2,-2},l}^s(I)$ with $s\geq 2$, then we have
\begin{align*}
&|\lambda^m_{N}(\tau_l)-\lambda^m(\tau_l)|\lesssim C(m)N^{2(2-s)}\max_{i=1,\cdots,m}(\|\partial_t^su_l^i\|_{\omega^{-2+s,-2+s}}+\|\partial_t^2u_l^i\|)^2 (s\leq3),\\
&|\lambda^m_{N}(\tau_l)-\lambda^m(\tau_l)|\lesssim C(m)N^{2(2-s)}\max_{i=1,\cdots,m}\|\partial_t^su_l^i\|_{\omega^{-2+s,-2+s}}^2 (s>3),
\end{align*}
where $C(m)$ is a constant independent of $N$.
\end{theorem}
\begin{proof}
For any $v\in E_{1,m}$ it can be represented by $v=\sum_{i=1}^m\mu_iu_l^i$, we then have
\begin{align*}
&\frac{B(v,v)-B(\Pi_{N}^{2,l}v,\Pi_{N}^{2,l}v)}{B(v,v)}\leq \frac{2|B(v,v-\Pi_{N}^{2,l}v)|}{B(v,v)}\\
&\leq\frac{2\sum_{i,j=1}^{m}|\mu_i| |\mu_j||B(u_l^i-\Pi_{N}^{2,l}u_l^i,u_l^j)|}{\sum_{i=1}^{m}|\mu_i|^2}\\
&\leq 2m\max_{i,j=1,\cdots,m}|B(u_l^i-\Pi_{N}^{2,l}u_l^i,u_l^j)|.
\end{align*}
From Cauchy-Schwarz inequality we have
\begin{align*}
&|B(u_l^i-\Pi_{N}^{2,l}u_l^i,u_l^j)|=\frac{1}{\lambda^j(\tau_l)}|\lambda^j(\tau_l)B(u_l^j,u_l^i-\Pi_{N}^{2,l}u_l^i)|\\
&=\frac{1}{\lambda^j(\tau_l)}|A_{\tau_l}(u_l^j,u_l^i-\Pi_{N}^{2,l}u_l^i)|
=\frac{1}{\lambda^j(\tau_l)}|A_{\tau_l}(u_l^j-\Pi_{N}^{2,l}u_l^j,u_l^i-\Pi_{N}^{2,l}u_l^i)|\\
&\leq\frac{1}{\lambda^j(\tau_l)}(A_{\tau_l}(u_l^j-\Pi_{N}^{2,l}u_l^j,u_l^j-\Pi_{N}^{2,l}u_l^j))^{\frac{1}{2}}(A_{\tau_l}(u_l^i-\Pi_{N}^{2,l}u_l^i,u_l^i-\Pi_{N}^{2,l}u_l^i))^{\frac{1}{2}}.\\
\end{align*}
 From Hardy inequality (cf. B 8.6 in \cite{shen2011}) we have
\begin{align*}
&\int_I\frac{1}{(t+1)^2}u^2dt\lesssim\int_I(\partial_tu)^2dt.
\end{align*}
Then from  Poincar$\acute{e}$ inequality we can obtain
\begin{align*}
&\|u\|_{2,\omega,l}^2=\int_I(t+1)^2(\partial_t^2u)^2dt+\int_I(\partial_tu)^2dt
+\int_I\frac{1}{(t+1)^2}u^2dt\\
&\lesssim \int_I(\partial_t^2u)^2dt+\int_I(\partial_tu)^2dt\lesssim \int_I(\partial_t^2u)^2dt.
\end{align*}
From the property of orthogonal project \eqref{project} and continuity of $A_{\tau_l}(u,v)$ in Theorem 1 we can derive that
\begin{align*}
&|B(u_l^i-\Pi_{N}^{2,l}u_l^i,u_l^j)|=\frac{1}{\lambda^j(\tau_l)}|A_{\tau_l}(u_l^j,u_l^i-\Pi_{N}^{2,l}u_l^i)|\\
&\leq\frac{1}{\lambda^j(\tau_l)}(A_{\tau_l}(u_l^j-\Pi_{N}^{2,l}u_l^j,u_l^j-\Pi_{N}^{2,l}u_l^j))^{\frac{1}{2}}(A_{\tau_l}(u_l^i-\Pi_{N}^{2,l}u_l^i,u_l^i-\Pi_{N}^{2,l}u_l^i))^{\frac{1}{2}}\\
&\leq\frac{1}{\lambda^j(\tau_l)}(A_{\tau_l}(u_l^j-\pi_{N}^{2,l}u_l^j,u_l^j-\pi_{N}^{2,l}u_l^j))^{\frac{1}{2}}(A_{\tau_l}(u_l^i-\pi_{N}^{2,l}u_l^i,u_l^i-\pi_{N}^{2,l}u_l^i))^{\frac{1}{2}}\\
&\lesssim\frac{1}{\lambda^j(\tau_l)}\|u_l^j-\pi_{N}^{2,l}u_l^j\|_{2,\omega,l}\|u_l^i-\pi_{N}^{2,l}u_l^i\|_{2,\omega,l}\\
&\lesssim \frac{1}{\lambda^j(\tau_l)}\|\partial_t^2(u_l^j-\pi_{N}^{2,l}u_l^j)\|\cdot\|\partial_t^2(u_l^i-\pi_{N}^{2,l}u_l^i)\|.
\end{align*}
Since
\begin{align*}
\frac{B(v,v)}{B(\Pi_{N}^{2,l}v,\Pi_{N}^{2,l}v)}\leq\frac{1}{1-2m\max_{i,j=1,\cdots,m}|B(u_l^i-\Pi_{N}^{2,l}u_l^i,u_l^j)|},
\end{align*}
then from Theorem 3 and Theorem 4 we can get desired results. $\Box$
\end{proof}

\section{Efficient implementation of the algorithm}
We describe in this section how to solve the TE mode \eqref{TE1}-\eqref{TE4} and TM mode \eqref{TM1}-\eqref{TM6}  efficiently.
 \subsection{Efficient implementation of the algorithm for TE mode}
 We start by constructing a set of basis functions for $X_N$.
Let
\begin{align*}
 \phi_i(t) =d_i(L_i(t)-\frac{2(2i+5)}{2i+7}L_{i+2}(t)+\frac{2i+3}{2i+7}L_{i+4}(t)), i = 0\cdots N-4,
\end{align*}
where $d_i=\frac{1}{\sqrt{2(2i+3)^2(2i+5)}}$ and $L_i$ is the Legendre polynomial of degree $i$.\\
 It is clear that
 \begin{align*}
 X_N=\text{span}\{\phi_0(t),\cdots,\phi_{N-4}(t)\}\oplus \text{span}\{\phi_{N-3}(t)\},
 \end{align*}
 where $\phi_{N-3}(t)=\frac{1}{4}(t+1)(t-1)^2$.
 Setting
   \begin{align*}
a_{i,j}=A_{\tau_l}(\phi_j(t),\phi_i(t)),\quad b_{i,j}=B(\phi_j(t),\phi_i(t)).
 \end{align*}
 We shall look for
  \begin{align}
u_{lN}=\sum_{i=0}^{N-3}u_{i}^l\phi_i.\label{express2}
 \end{align}
 Now, plugging the expression of \eqref{express2} in \eqref{discrete1}, and taking $v_{N}$
through all the basis functions in $X_N$, we will arrive at the following linear eigenvalue
system:
 \begin{align}
AU_l=\lambda_{N}(\tau_l)BU_l, \label{5.2}
\end{align}
where
 \begin{align*}
&A=(a_{ij}),B=(b_{ij}),U_l=(u_{0}^l,\cdots,u_{N-3}^l)^T.
\end{align*}

 \subsection{Efficient implementation of the algorithm for TM mode}
 Let $r=\frac{t+1}{2}R$, $W(t)=w_l^m(\frac{t+1}{2}R)$,$V(t)=v_l^m(\frac{t+1}{2}R)$,$\tilde{n}(t)=n(\frac{t+1}{2}R)$, $w(t)=\bar{w}_l^m(\frac{t+1}{2}R)$,$v(t)=\bar{v}_l^m(\frac{t+1}{2}R)$, $D_l^{\pm}=\partial_t\pm\frac{l}{l+1}$.
Then  \eqref{TM1}-\eqref{TM6} can be restated as follows:
 \begin{align}
&D_l^-(D_{l-1}^-W(t)-D_{l+2}^+V(t))-k_l^2\frac{R^2}{4}\tilde{n}(t)\frac{2l+1}{l}V(t)=0,&t\in(-1,1), \label{resTM1}\\
&D_{l+1}^+(D_{l+2}^+V(t)-D_{l-1}^-W(t))-k_l^2\frac{R^2}{4}\tilde{n}(t)\frac{2l+1}{l+1}W(t)=0,&t\in(-1,1),\label{resTM2}\\
&D_l^-(D_{l-1}^-w(t)-D_{l+2}^+v(t))-k_l^2\frac{R^2}{4}\frac{2l+1}{l}v(t)=0, &t\in(-1,1),\label{resTM3}\\
&D_{l+1}^+(D_{l+2}^+v(t)-D_{l-1}^-w(t))-k_l^2\frac{R^2}{4}\frac{2l+1}{l+1}w(t)=0,&t\in(-1,1),\label{resTM4}\\
&W(1)-V(1)=w(1)-v(1),\label{resTM5}\\
&D_{l-1}^-W(1)-D_{l+2}^+V(1)=D_{l-1}^-w(1)-D_{l+2}^+v(1).\label{resTM6}
\end{align}
Let us define
\begin{align*}
\mathbf{H}(I)=\{\mathbf{h}=(h_1(t),h_2(t),h_3(t),h_4(t))\in (H^1(I))^4:h_1(1)-h_2(1)=h_3(1)-h_4(1)\}.
 \end{align*}
Then the weak form of \eqref{resTM1}-\eqref{resTM6} is: Find $(k_l^2,(W(t),V(t),w(t),v(t))\neq\mathbf{0})\in\mathbb{C}\times \mathbf{H}(I),$ such that for all $(h_1(t),h_2(t),h_3(t),h_4(t))\in \mathbf{H}(I)$,
 \begin{align}
&\int_I(D_{l+2}^+V(t)-D_{l-1}^-W(t))(D_{l+2}^+h_1(t)-D_{l-1}^-h_2(t))(t+1)^2dt \nonumber\\
&-\int_I(D_{l+2}^+v(t)-D_{l-1}^-w(t))(D_{l+2}^+h_3(t)-D_{l-1}^-h_4(t))(t+1)^2dt \nonumber\\
&=k_l^2(2l+1)\frac{R^2}{4}(\frac{1}{l}\int_I(\tilde{n}(t)V(t)h_1(t)-v(t)h_3(t))(t+1)^2dt\nonumber\\
&+\frac{1}{l+1}\int_I(\tilde{n}(t)W(t)h_2(t)-w(t)h_4(t))(t+1)^2dt). \label{weaktm}
 \end{align}
 Let $\mathbf{H}_N=\mathbf{H}(I)\cap P_N^4$, then the discrete form of \eqref{weaktm} is: Find $(k_{Nl}^2,(W_N,V_N,w_N,v_N)\neq\mathbf{0})\in\mathbb{C}\times \mathbf{H}_N,$ such that for all $(h_{1N}(t),h_{2N}(t),h_{3N}(t),h_{4N}(t))\in \mathbf{H}_N$,
 \begin{align}
&\int_I(D_{l+2}^+V_N(t)-D_{l-1}^-W_N(t))(D_{l+2}^+h_{1N}(t)-D_{l-1}^-h_{2N}(t))(t+1)^2dt \nonumber\\
&-\int_I(D_{l+2}^+v_N(t)-D_{l-1}^-w_N(t))(D_{l+2}^+h_{3N}(t)-D_{l-1}^-h_{4N}(t))(t+1)^2dt \nonumber\\
&=k_{Nl}^2(2l+1)\frac{R^2}{4}(\frac{1}{l}\int_I(\tilde{n}(t)V_N(t)h_{1N}(t)-v_N(t)h_{3N}(t))(t+1)^2dt\nonumber\\
&+\frac{1}{l+1}\int_I(\tilde{n}(t)W_N(t)h_{2N}(t)-w_N(t)h_{4N}(t))(t+1)^2dt). \label{discretetm}
 \end{align}
 We  now  construct a set of basis functions for $\mathbf{H}_N$. Let
 \begin{align*}
\varphi_i(t)=L_i(t)-L_{i+2}(t), (i=0,1,\cdots,N-2), \varphi_{N-1}=\frac{1-t}{2},\varphi_N=1.
 \end{align*}
Setting
 \begin{align*}
&\mathbf{\psi}_j(t)=(\varphi_j(t),0,0,0),\mathbf{\psi}_{N+j}(t)=(0,\varphi_j(t),0,0),\mathbf{\psi}_{2N+j}(t)=(0,0,\varphi_j(t),0),\\
&\mathbf{\psi}_{3N+j}(t)=(0,0,0,\varphi_j(t)),\mathbf{\psi}_{4N}(t)=(\varphi_N(t),0,\varphi_N(t),0),\\
&\mathbf{\psi}_{4N+1}(t)=(\varphi_N(t),\varphi_N(t),0,0),\mathbf{\psi}_{4N+2}(t)=(0,0,\varphi_N(t),\varphi_N(t)),
  \end{align*}
  where $j=0,1,\cdots,N-1$.
It is clear that
 \begin{align*}
 \mathbf{H}_N=\text{span}\{\mathbf{\psi}_0(t),\mathbf{\psi}_1(t),\cdots,\mathbf{\psi}_{4N+2}(t)\}.
 \end{align*}
 We shall look for
  \begin{align}
(W,V,w,v)=\sum_{i=0}^{4N+2}\alpha_{i}\psi_i(t).\label{express3}
 \end{align}
 Now, plugging the expression of \eqref{express3} in \eqref{discretetm}, and taking $(h_{1N}(t),h_{2N}(t),h_{3N}(t),h_{4N}(t))$
through all the basis functions in $\mathbf{H}_N$, we will arrive at the following linear eigenvalue
system:
 \begin{align}
 \mathcal{A}\mathcal{U}=k_{Nl}^2\mathcal{B}\mathcal{U}, \label{5.12}
\end{align}
where
$\mathcal{U}=(\alpha_0,\alpha_1,\cdots,\alpha_{4N+2})^T,$
$ \mathcal{A} $ and $ \mathcal{B}$  are the corresponding  stiff matrix and mass matrix, respectively .\\
Note that the stiff matrix $A$ (resp.~$\mathcal{A}$)  and the mass matrix $B$ (resp.~$\mathcal{B}$) in \eqref{5.2} (resp.~\eqref{5.12}) are all sparse for the constant $n$. Thus either direct or iterative eigen solvers can be efficiently applied in parallel. For the case of variable $n$, since one is mostly interested in a few smallest transmission eigenvalues, it is
most efficient to solve \eqref{5.2} (resp.~\eqref{5.12}) by using a shifted inverse power method (cf., for instance,
\cite{Go1989}) which requires solving, repeatedly for different righthand side $f$ (resp. $\tilde{f}$),

 \begin{align}
(A-\lambda_a(\tau_l)B)U_l=f~(\text{resp.} ~~(\mathcal{A}-k_{al}^2\mathcal{B})\mathcal{U}=\tilde{f}),
\end{align}
where $\lambda_a(\tau_l)$ (resp. $k_{al}$) is some approximate value for the transmission eigenvalue $\lambda_N(\tau_l)$ (resp. $k_{Nl}$).
The above system can be efficiently solved by the Schur-complement approach,
 we refer to   \cite{kw2007} for a detailed description on a
related problem.
In summary, the approximate transmission eigenvalue problem \eqref{5.2} (resp.~\eqref{5.12}) can be solved
very efficiently.
\section{Numerical experiments}
\indent We now perform a sequence of numerical tests to study the convergence behavior
and show the effectiveness of our algorithm. We operate our programs in MATLAB 2015b.
\subsection{Homogeneous medium $n$}
We take $R=1$, $n=16$, and $l=1,2,3$ in our examples. Numerical results for the first four eigenvalues with different $l$ and $N$ are listed in Tables 6.1-6.3 for the TE mode and in Tables 6.4-6.6 for the TM mode, respectively.

\begin{center}{\small\bf Table 6.1 The first four eigenvalues of the TE mode for $l=1$ and different $N$ in unit ball}
\begin{tabular}{{ccccc}}
  \hline
\multicolumn{1}{c} N &$1st$&      $2nd$&    $ 3rd$&   $4th$\\ \hline
\multicolumn{1}{c}{10}&      1.460855902327352 &  2.309268980991891&  3.142098536003481 &   4.004131427018431\\
\multicolumn{1}{c}{15}&     1.460842273223355 &  2.309270674683650&   3.141592652865679&   4.028313168694923\\
\multicolumn{1}{c}{20}&    1.460855902076009&    2.309270674683547&   3.141592653589792&  4.028312376370235\\
\multicolumn{1}{c}{25}&  1.460855902076010&    2.309270674683545&      3.141592653589794&  4.028312376370704\\
\multicolumn{1}{c}{30}&  1.460855902076010&   2.309270674683548&      3.141592653589792&   4.028312376370695\\
  \hline
\end{tabular}\end{center}
\begin{center}{\small\bf Table 6.2 The first four eigenvalues of the TE mode for $l=2$ and different $N$
in unit ball }
\begin{tabular}{{ccccc}}
  \hline
\multicolumn{1}{c} N&$1st$&      $2nd$&    $ 3rd$&   $4th$\\ \hline
\multicolumn{1}{c}{10}&    1.764042417090797&   2.631690641913292&  3.465152456664587& 4.304027448259460\\
\multicolumn{1}{c}{15}&    1.764042422029338&   2.631678257808169& 3.465236228216971& 4.293583051636833\\
\multicolumn{1}{c}{20} & 1.764042422029338&    2.631678257809425&     3.465236224179554&  4.293582919867683\\
\multicolumn{1}{c}{25} &  1.764042422029339&    2.631678257809425&      3.465236224179551&  4.293582919866948\\
\multicolumn{1}{c}{30} & 1.764042422029338&    2.631678257809422&     3.465236224179552&  4.293582919866945\\
\\
  \hline
\end{tabular}\end{center}
\begin{center}{\small\bf Table 6.3 The first four eigenvalues of the TE mode for $l=3$ and different $N$
in unit ball}
\begin{tabular}{{ccccc}}
  \hline
\multicolumn{1}{c} N &$1st$&      $2nd$&    $ 3rd$&   $4th$\\ \hline
\multicolumn{1}{c}{10} &   2.061050417316723  &  2.949614759058087& 3.789381768848592&  4.729312472707909\\
\multicolumn{1}{c}{15} &   2.061050433015994 &   2.949488215613814& 3.792296568087934&   4.619875856328957\\
\multicolumn{1}{c}{20} &   2.061050433015994&   2.949488215659479&   3.792296458205384& 4.619887058309071\\
\multicolumn{1}{c}{25} &   2.061050433015994&   2.949488215659483&    3.792296458205412&  4.619887058253892\\
\multicolumn{1}{c}{30} &   2.061050433015993&    2.949488215659481&    3.792296458205412&     4.619887058253892\\
  \hline
\end{tabular}\end{center}

\begin{center}{\small\bf Table 6.4 The first four eigenvalues of the TM mode for $l=1$ and different $N$
in unit ball}
\begin{tabular}{{ccccc}}
  \hline
\multicolumn{1}{c} N &$1st$&      $2nd$&    $ 3rd$&   $4th$\\ \hline
\multicolumn{1}{c}{10}&    1.165407223825574&    2.045867252670490  & 3.158017632244895& 3.442289242670909\\
\multicolumn{1}{c}{15}&     1.165407223827102 &  2.045867782103358&   3.418097647742193&  4.292461561664239\\
\multicolumn{1}{c}{20}&   1.165407223827106&    2.045867782103363&  3.418097651533288&  4.292488875027842\\
\multicolumn{1}{c}{25}&   1.165407223827098& 2.045867782103357&      3.418097651533263&  4.292488875029357\\
\multicolumn{1}{c}{30}&  1.165407223827108&  2.045867782103361&      3.418097651533326&   4.292488875029386\\
  \hline
\end{tabular}\end{center}
\begin{center}{\small\bf Table 6.5 The first four eigenvalues of the TM mode for $l=2$ and different $N$
in unit ball }
\begin{tabular}{{ccccc}}
  \hline
\multicolumn{1}{c} N&$1st$&      $2nd$&    $ 3rd$&   $4th$\\ \hline
\multicolumn{1}{c}{10}&     1.475116524235235&   2.340653677153939&  3.231276889331592& 3.354682219133095\\
\multicolumn{1}{c}{15}&     1.475116524493845&   2.340657592735246&  3.233313705482311& 4.557043549625480\\
\multicolumn{1}{c}{20} & 1.475116524493846&    2.340657592735365&      3.233313708702767& 4.557097304644416\\
\multicolumn{1}{c}{25} &   1.475116524493844&  2.340657592735367&      3.233313708702756& 4.557097304725269\\
\multicolumn{1}{c}{30} & 1.475116524493843&   2.340657592735366&    3.233313708702761&  4.557097304725263\\
\\
  \hline
\end{tabular}\end{center}
\begin{center}{\small\bf Table 6.6 The first four eigenvalues of the TM mode for $l=3$ and different $N$
in unit ball}
\begin{tabular}{{ccccc}}
  \hline
\multicolumn{1}{c} N &$1st$&      $2nd$&    $ 3rd$&   $4th$\\ \hline
\multicolumn{1}{c}{10} & 1.777410985980160 &  2.656258796160595& 3.455344578108596& 3.582237584322505\\
\multicolumn{1}{c}{15} &  1.777410996101287 &    2.656264636190942& 3.512014047361030&  4.421816274456907\\
\multicolumn{1}{c}{20} &   1.777410996101286&    2.656264636197187&   3.512014051598614&4.421843661628902\\
\multicolumn{1}{c}{25} &  1.777410996101286&   2.656264636197187&    3.512014051598621&   4.421843661635361\\
\multicolumn{1}{c}{30} &  1.777410996101284&    2.656264636197187&    3.512014051598614&  4.421843661635358\\
  \hline
\end{tabular}\end{center}
We see from Tables 6.1-6.6 that numerical eigenvalues achieve at least fourteen-digit accuracy with $N\geq 25$.
As a comparison, we list in Table 6.7 the results in \cite{FE}  computed by the
 finite element method with mesh size $h\thickapprox 0.2$. We observe that
numerical eigenvalues listed in Table 6.7 have only two- to three-digit accuracy, while our method
captured at least fourteen-digit accuracy with much less computational efforts.

Note that the multiplicity of the eigenvalue can be predicted by $|m|\le l$.
For $l=1$, we have $m=0,\pm 1$ and hence the first eigenvalue has multiplicity 3; for $l=2$, we have $m=0,\pm 1, \pm 2$
and therefore the second eigenvalue has multiplicity 5. In general, the $j$-th eigenvalue has multiplicity $2j+1$.
The first three numbers (read in horizontal) in Table 6.7 are for the TM mode with $l=1$, which compared with the first column in Table 6.4;
the second three numbers are for the TE mode with $l=1$, which compared with the first column in Table 6.1; the next five numbers are for the TM mode with $l=2$, which compared with the first column in Table 6.5; the following five numbers are for the TM mode with $l=2$, which compared with the
first column in Table 6.2; the last seven numbers in Table 6.7 are for the TE mode with $l=3$, which compared with the first column in Table 6.6.

\begin{center}{\small\bf Table 6.7 Computed Maxwell's transmission eigenvalues for the unit ball}
\begin{tabular}{{cccccccc}}
  \hline
 1.1741 &1.1717&1.1721&1.4665&1.4667&1.4671&1.4824&1.4828\\
 1.4828&1.4830&1.4836&1.7690&1.7690&1.7698&1.7700&1.7705\\
 1.7857&1.7859&1.7862&1.7865&1.7867&1.7868&1.7872&   \\
 \hline
\end{tabular}\end{center}

In order to further demonstrate accuracy and efficiency of our algorithm, we use numerical solutions with $N = 60$ as reference solutions
and plot errors of approximated eigenvalues with different $N$ in Figures 1-6. We observe exponential decay in all cases.

\begin{figure}[h!]
\begin{minipage}[c]{0.48\textwidth}
\centering
\includegraphics[width=1\textwidth]{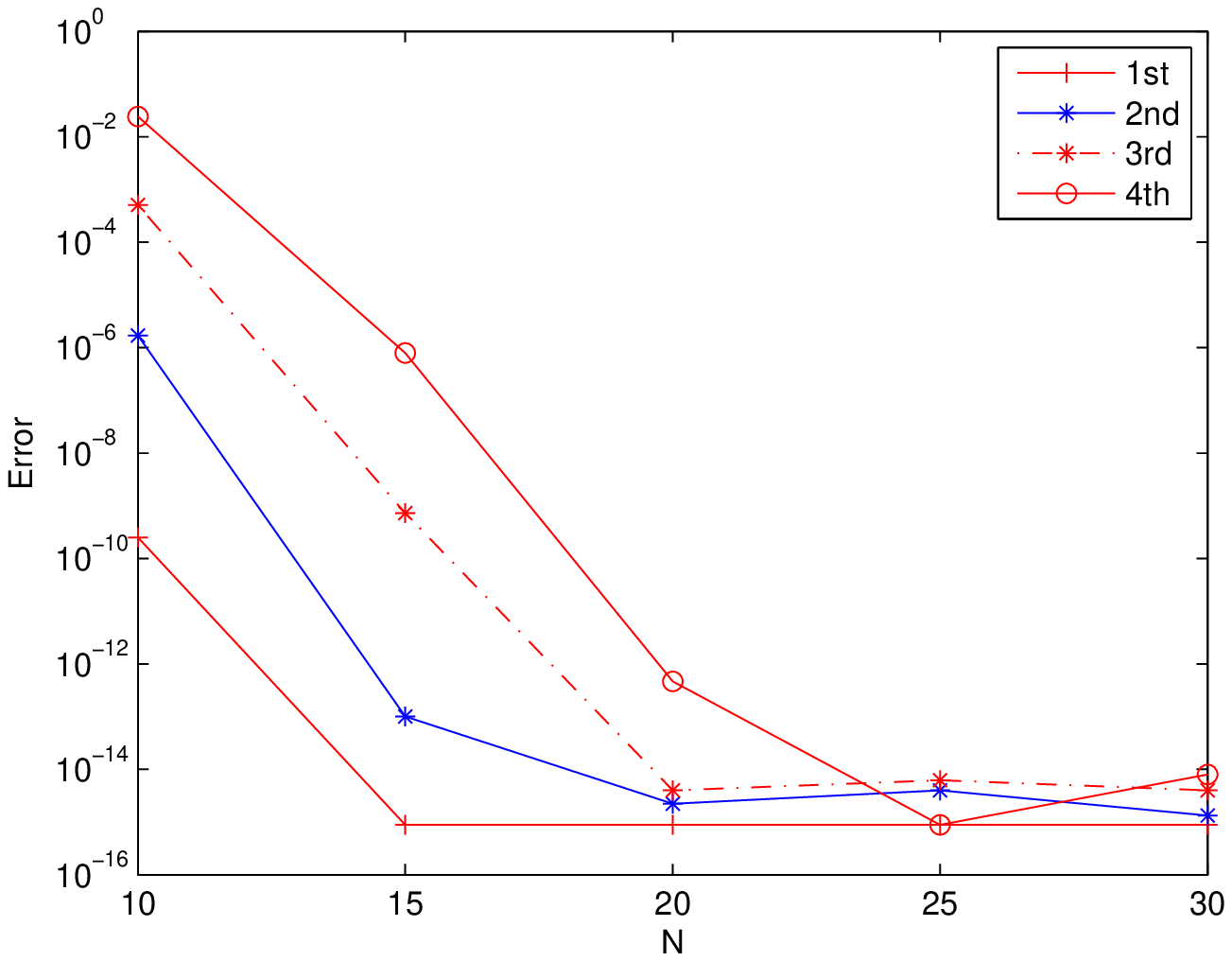}
\caption{Errors between numerical and reference solution of the TE mode for $l=1$.}\label{fig1}
\end{minipage}\hfill%
\begin{minipage}[c]{0.48\textwidth}
\centering
\includegraphics[width=1\textwidth]{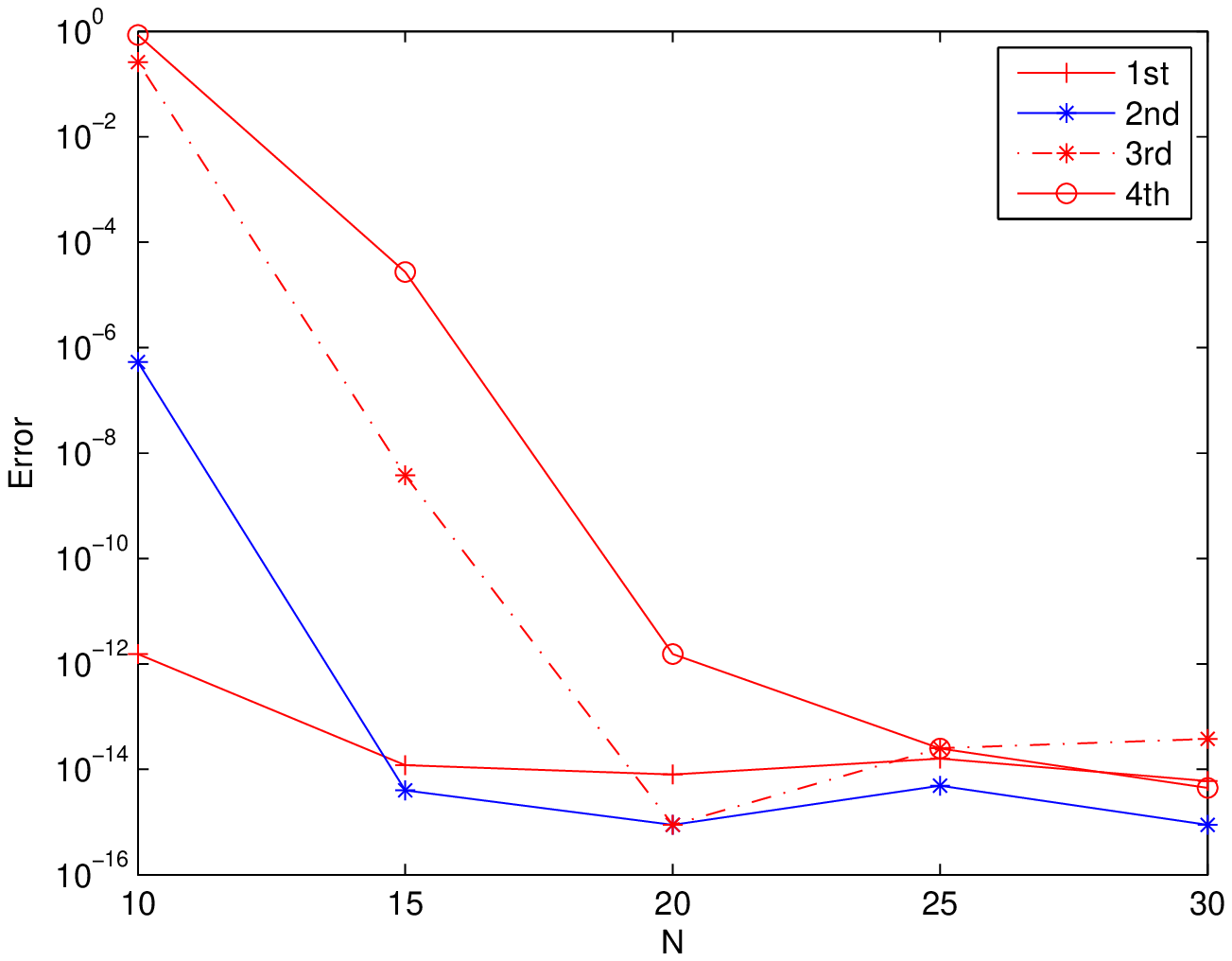}
\caption{Errors between numerical and reference solutions of the TM mode for $l=1$.}\label{fig2}
\end{minipage}\hfill%
\end{figure}

\begin{figure}[h!]
\begin{minipage}[c]{0.48\textwidth}
\centering
\includegraphics[width=1\textwidth]{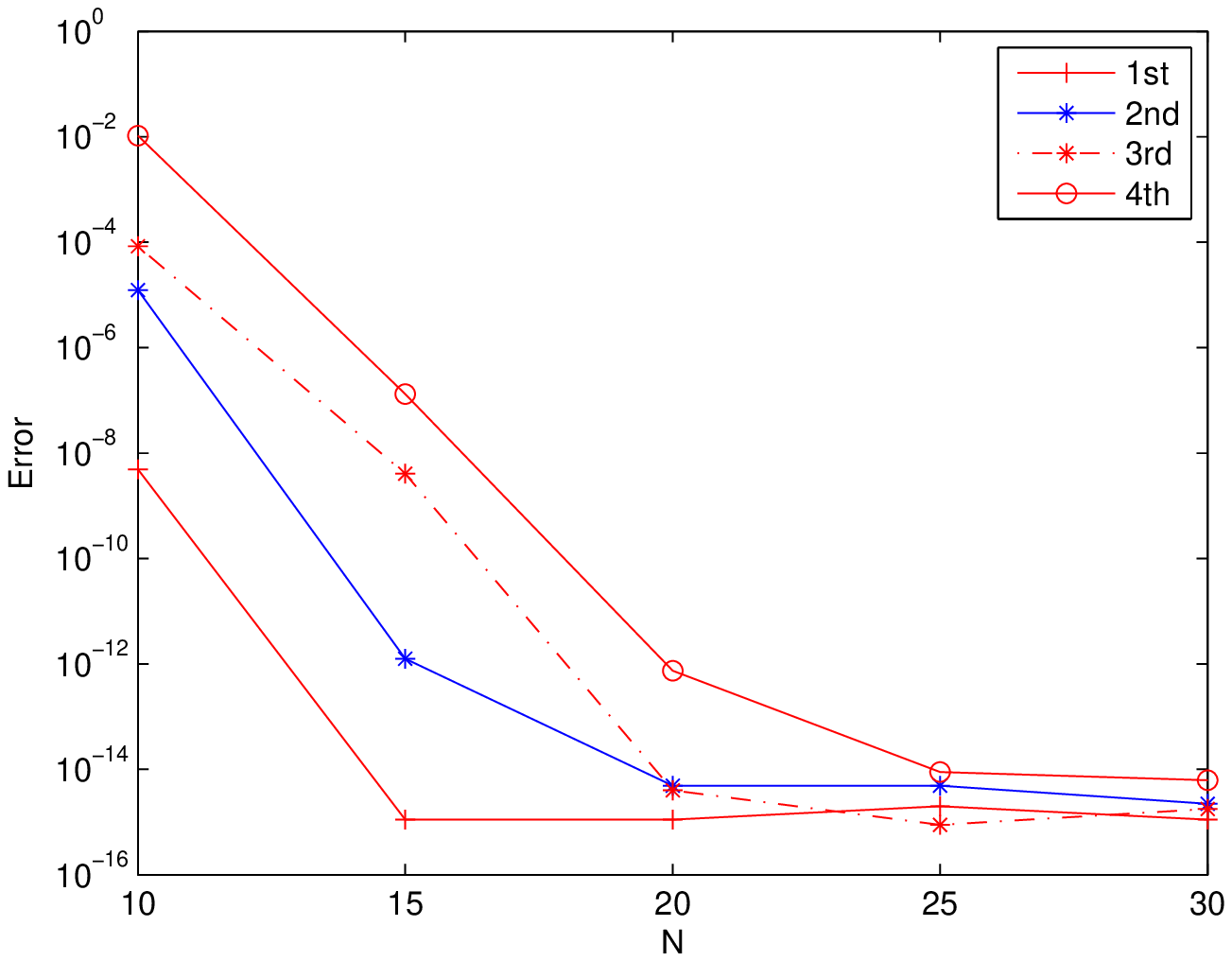}
\caption{Errors between numerical and reference solutions of the TE mode for $l=2$.}\label{fig3}
\end{minipage}\hfill%
\begin{minipage}[c]{0.48\textwidth}
\centering
\includegraphics[width=1\textwidth]{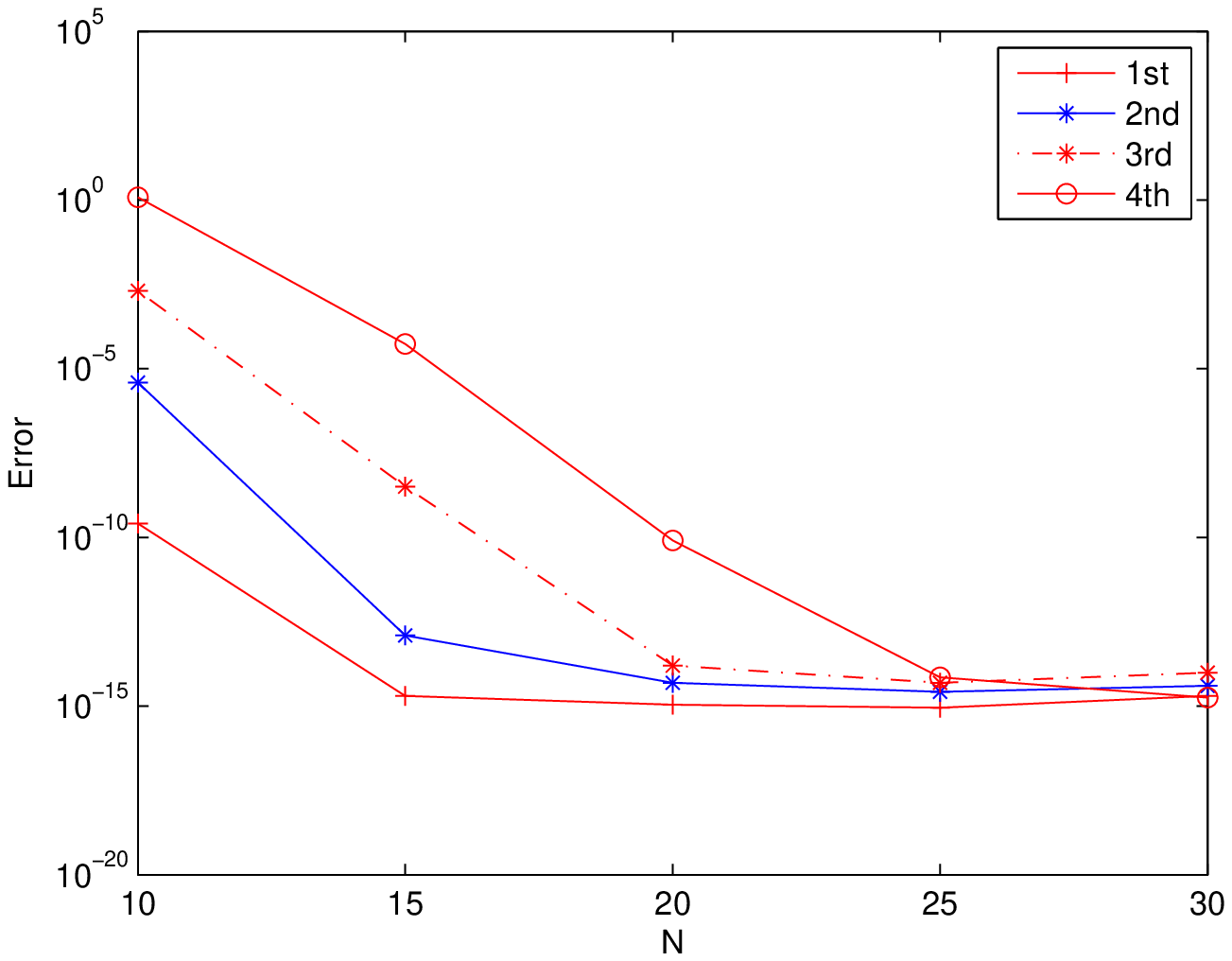}
\caption{Errors between numerical and reference solutions of the TM mode for $l=2$.}\label{fig4}
\end{minipage}\hfill%
\end{figure}

\begin{figure}[h!]
\begin{minipage}[c]{0.48\textwidth}
\centering
\includegraphics[width=1\textwidth]{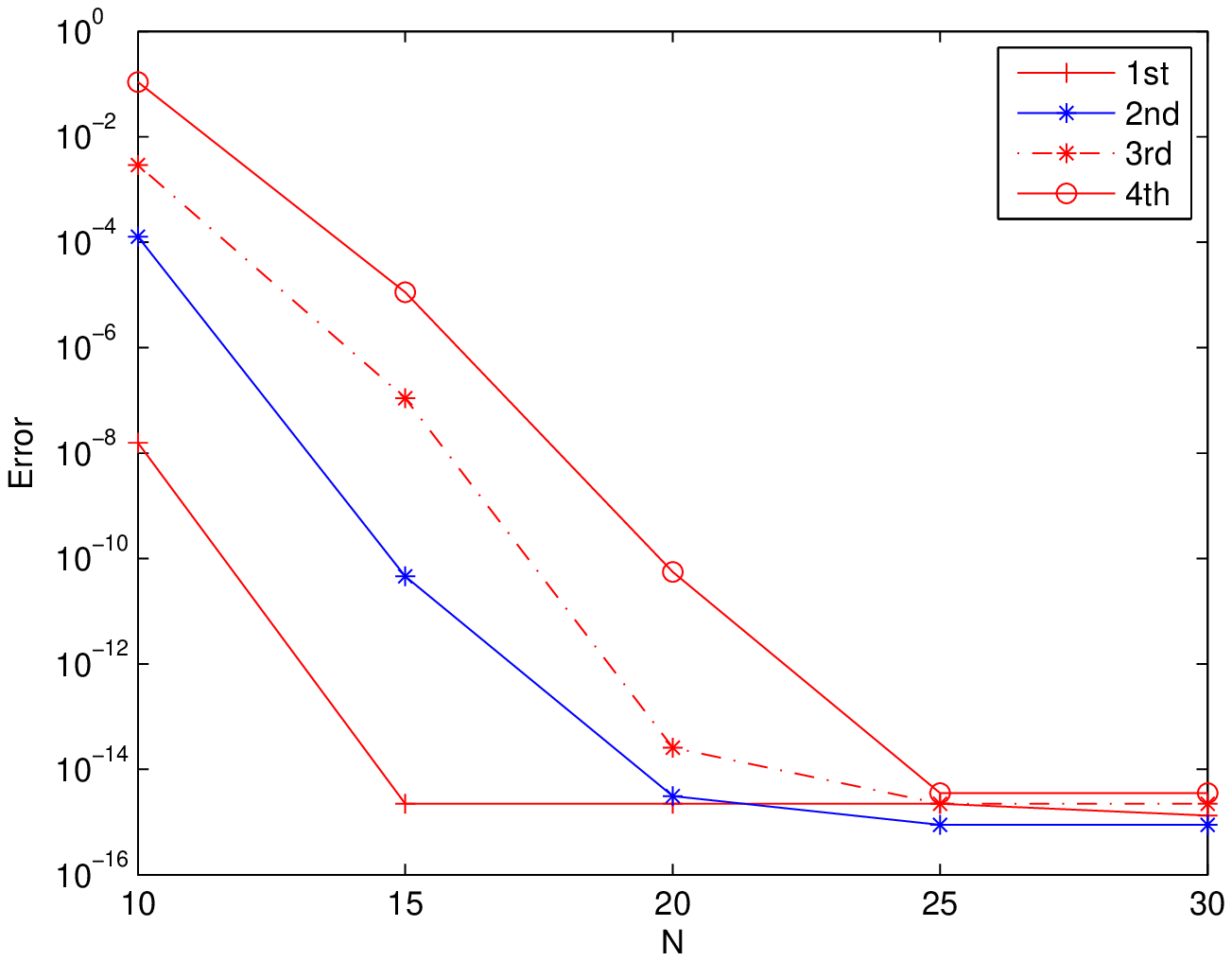}
\caption{Errors between numerical and reference solutions of the TE mode for $l=3$.}\label{fig5}
\end{minipage}\hfill%
\begin{minipage}[c]{0.48\textwidth}
\centering
\includegraphics[width=1\textwidth]{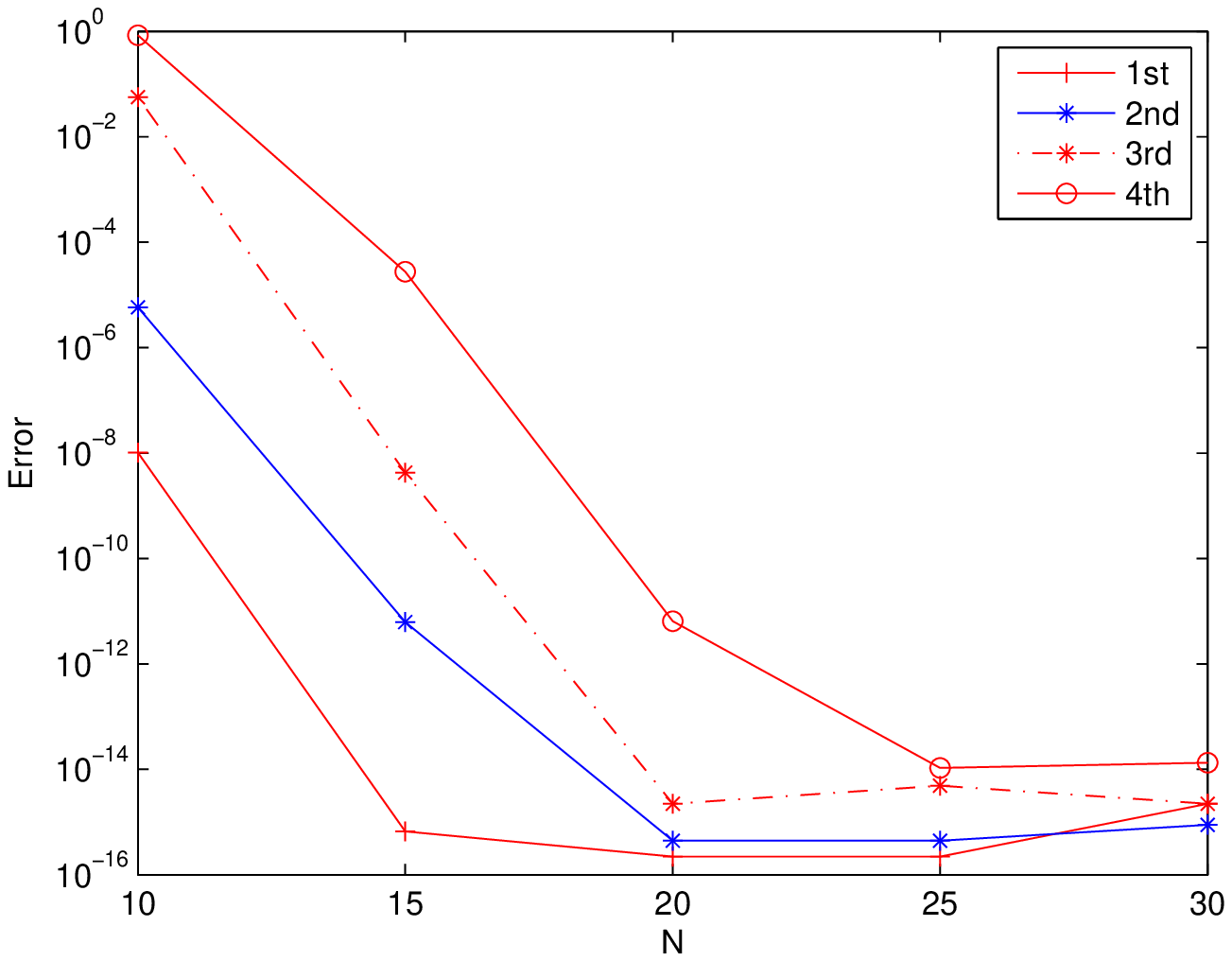}
\caption{Errors between numerical and reference solutions of the TM mode for $l=3$.}\label{fig6}
\end{minipage}\hfill%
\end{figure}

\subsection{Inhomogeneous medium $n$ }
We take $R=1$, $n=8+4r^2$, and $l=1$ in our examples. Numerical results of the first four eigenvalues with different $N$ are listed in Table 6.8 for the TE mode and in Table 6.9 for the TM mode, respectively.
\begin{center}{\small\bf Table 6.8 The first four eigenvalues of the TE mode for $l=1$ and different $N$ in unit ball}
\begin{tabular}{{ccccc}}
  \hline
\multicolumn{1}{c} N &$1st$&      $2nd$&    $ 3rd$&   $4th$\\ \hline
\multicolumn{1}{c}{10}&      1.924760241224309 &  3.066320509754762 &  4.953581565194202&  6.219217568046522\\
\multicolumn{1}{c}{15}&      1.924760240239596 &  3.066318451356159&   4.944962746965171&  6.162042706007884\\
\multicolumn{1}{c}{20}&   1.924760240239595&      3.066318451356097&  4.944962719618220&  6.162013704025296\\
\multicolumn{1}{c}{25}&   1.924760240239596&  3.066318451356097&     4.944962719618172&    6.162013703949516\\
\multicolumn{1}{c}{30}&   1.924760240239597&  3.066318451356096&       4.944962719618174&    6.162013703949522\\
  \hline
\end{tabular}\end{center}
\begin{center}{\small\bf Table 6.9 The first four eigenvalues of the TM mode for $l=1$ and different $N$ in unit ball}
\begin{tabular}{{ccccc}}
  \hline
\multicolumn{1}{c} N &$1st$&      $2nd$&    $ 3rd$&   $4th$\\ \hline
\multicolumn{1}{c}{10}&    1.546722576754737 &3.418052693123285 & 4.366864963075772 &4.640915773722154   \\
\multicolumn{1}{c}{15}&    1.546722576768443 &3.418109299464758 &  4.616102608978945 &6.423176009030875   \\
\multicolumn{1}{c}{20}&   1.546722576768448 & 3.418109299467635& 4.616102624493460 & 6.425723290604534\\
\multicolumn{1}{c}{25}&  1.546722576768440 & 3.418109299467642 & 4.616102624493454 & 6.425723292013815\\
\multicolumn{1}{c}{30}& 1.546722576768443 &3.418109299467622 &4.616102624493481 &6.425723292013920\\
  \hline
\end{tabular}\end{center}
We see from Tables 6.8-6.9 that numerical eigenvalues achieve at least thirteen-digit accuracy with $N\geq 25$.
Here again, we use numerical solutions with $N = 60$ as reference solutions and plot the errors in Figure 7-8.

\begin{figure}[h!]
\begin{minipage}[c]{0.48\textwidth}
\centering
\includegraphics[width=1\textwidth]{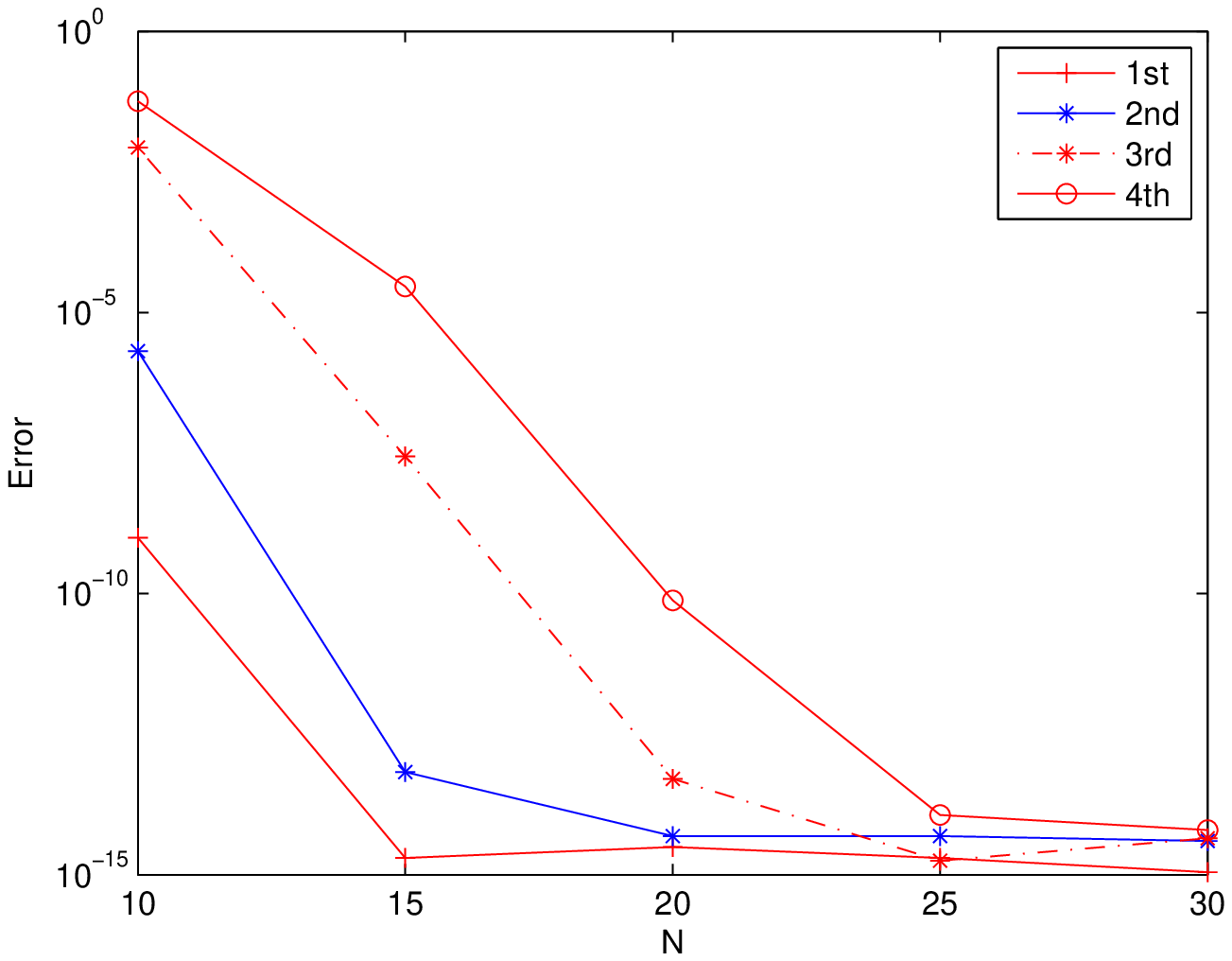}
\caption{Errors between numerical and reference solutions of the TE mode for $l=1$.}\label{fig7}
\end{minipage}\hfill%
\begin{minipage}[c]{0.48\textwidth}
\centering
\includegraphics[width=1\textwidth]{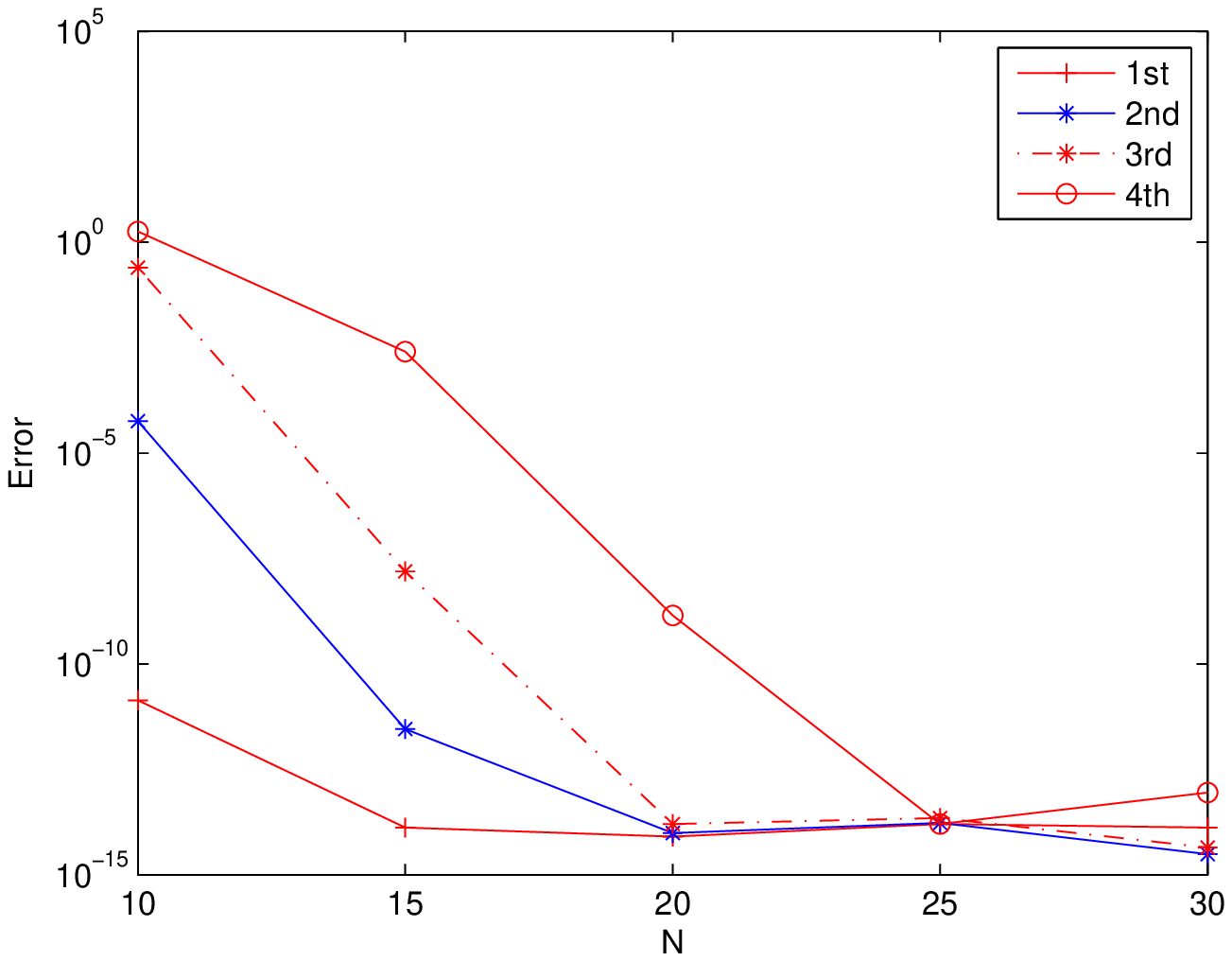}
\caption{Errors between numerical and reference solutions of the TM mode for $l=1$.}\label{fig8}
\end{minipage}\hfill%
\end{figure}
\section{Conclusions}\label{conc}
In summary, the method developed
in this paper is a first but important step towards a robust and accurate algorithm for more general transmission eigenvalue
problems which will be the subject of our future work.


\end{document}